\documentclass[journal,twoside,web]{ieeecolor}
\usepackage{generic}
\usepackage{cite}
\usepackage{amsmath,amssymb,amsfonts}
\usepackage{algorithmic}
\usepackage{graphicx}
\usepackage{textcomp}
\usepackage{hyperref}
\usepackage{mathrsfs}
\usepackage{comment}
\usepackage{cancel,ulem,booktabs,multirow}
\usepackage{mathtools}
\mathtoolsset{showonlyrefs=true}

\usepackage[ruled,vlined]{algorithm2e}

\usepackage{amsthm}

\newcommand{\Prob}{\mathbb{P}}
\newcommand{\R}{\mathbb{R}}

\newcommand{\F}{\mathscr{F}}
\newcommand{\A}{\mathscr{A}}
\newcommand{\E}{\mathbb{E}}

\newcommand{\ud}{\,\mathrm{d}}
\newcommand{\QV}[2]{\left\langle #1, #2\right\rangle}

\newcommand{\BAR}[1]{\overline{#1}}

\newcommand{\REG}{\mathrm{reg}}
\newcommand{\red}{\mathrm{red}}
\newcommand{\transpose}{^{\operatorname{T}}}

\newtheorem{theorem}{Theorem}
\newtheorem{lemma}{Lemma}
\newtheorem{prop}{Proposition}

\newtheorem{remark}{Remark}
\newtheorem{assu}{Assumption}
\newtheorem{cor}{Corollary}

\def\BibTeX{{\rm B\kern-.05em{\sc i\kern-.025em b}\kern-.08em
    T\kern-.1667em\lower.7ex\hbox{E}\kern-.125emX}}
\begin{document}

\title{Adversarial Decision-Making in Partially Observable Multi-Agent Systems: A Sequential Hypothesis Testing Approach}

\author{Haosheng Zhou, Daniel Ralston, Xu Yang and Ruimeng Hu
\thanks{This work was partially supported by the ONR grant under \#N00014-24-1-2432, the Simons Foundation (MP-TSM-00002783), and the NSF grant DMS-2109116 and DMS-2420988.}
\thanks{H. Zhou is with the Department of Statistics and Applied Probability, University of California, Santa Barbara, CA 93106, USA
        {\tt\small hzhou593@ucsb.edu}}%
\thanks{D. Ralston is with the Department of Mathematics, University of California, Santa Barbara, CA 93106, USA
        {\tt\small danielralston@ucsb.edu}}%
\thanks{X. Yang is with the Department of Mathematics, University of California, Santa Barbara, CA 93106, USA
        {\tt\small xy6@ucsb.edu}}%
\thanks{R. Hu is with the Department of Mathematics, and Department of Statistics and Applied Probability, University of California, Santa Barbara, CA 93106, USA
        {\tt\small rhu@ucsb.edu}}%
}

\maketitle

\begin{abstract}
Adversarial decision-making in partially observable multi-agent systems requires sophisticated strategies for both deception and counter-deception. This paper presents a sequential hypothesis testing (SHT)-driven framework that captures the interplay between strategic misdirection and inference in adversarial environments. We formulate this interaction as a partially observable Stackelberg game, where a follower agent (blue team) seeks to fulfill its primary task while actively misleading an adversarial leader (red team). In opposition, the red team, leveraging leaked information, instills carefully designed patterns to manipulate the blue team's behavior, mitigating the misdirection effect. Unlike conventional approaches that focus on robust control under adversarial uncertainty, our framework explicitly models deception as a dynamic optimization problem, where both agents strategically adapt their policies in response to inference and counter-inference. We derive a semi-explicit optimal control solution for the blue team within a linear-quadratic setting and develop iterative and machine learning-based methods to characterize the red team’s optimal response. Numerical experiments demonstrate how deception-driven strategies influence adversarial interactions and reveal the impact of leaked information in shaping equilibrium behaviors. These results provide new insights into strategic deception in multi-agent systems, with potential applications in cybersecurity, autonomous decision-making, and financial markets.
\end{abstract}

\begin{IEEEkeywords}
Adversarial decision-making, partially observable games, sequential hypothesis testing, Stackelberg games, stochastic optimal control, strategic deception and counter-deception.
\end{IEEEkeywords}

\section{Introduction}\label{sec:I}
Deception is a fundamental aspect of strategic interactions, shaping decision-making in adversarial settings across domains such as cybersecurity~\cite{aggarwal2016cyber}, financial markets~\cite{gerschlager2005deception}, and autonomous systems~\cite{arkin2011moral}. In these domains, adversarial decision-making plays a crucial role, with opposing agents employing strategies to outmaneuver and mislead one another (e.g., \cite{yager2008knowledge}). While the importance of deception has been recognized since ancient times, famously emphasized in Sun Zi’s \textit{The Art of War}~\cite{zi2007art}, modern applications involve adversarial agents that both infer their opponents’ intentions and mislead them in return. This paper examines such adversarial interactions through a red--blue team setting (e.g., \cite{rajendran2011blue}), with a particular focus on the interplay between sequential hypothesis testing and stochastic control.

In this work, we study adversarial decision-making in a partially observable environment where one agent (the red team) seeks to infer the objectives of its opponent (the blue team), while the blue team actively misdirects this inference process. 
Throughout the paper, the terms “red team” and “blue team” refer to two individual decision makers, with extensions to networked or multi-agent teams briefly discussed in the Conclusion.
Unlike conventional robust control frameworks that passively defend against adversarial influence, our approach explicitly models active deception, where the blue team deliberately introduces perturbations to shape the red team’s inference process. Anticipating this, the red team counters by strategically manipulating the blue team’s belief formation through leaked control information. This interplay results in a dynamic game of deception and counter-deception, where both teams exploit information asymmetry and adapt their strategies to outmaneuver one another.

Firstly, to formalize deception, we model such interaction using sequential hypothesis testing (SHT), a statistical method for dynamically evaluating hypotheses as data becomes available~\cite{tartakovsky2014sequential,goodman2007adaptive,schonbrodt2017sequential}. By integrating SHT into a linear-quadratic control framework, we capture the trade-offs between achieving primary objectives and engaging in strategic deception.
This allows the blue team to balance deception with fulfilling its primary task, while responding to the red team’s adaptive inference mechanisms.
Unlike filtering-based approaches~\cite{bain2009fundamentals,davis1977linear}, our model avoids explicit state estimation while still effectively capturing the complexities of partially observable adversarial interactions.
Furthermore, it differs from conventional robust control frameworks, which primarily focus on passive resilience rather than proactive misdirection~\cite{bauso2016robust,hakobyan2024wasserstein,moon2016linear,taskesen2024distributionally}. To the best of our knowledge, this is the first work that directly embeds test statistics into the cost functionals of control problems, explicitly modeling deception as a strategic component of decision-making.

Secondly, we formulate adversarial interaction as a leader-follower Stackelberg game, where the red team takes the leader's role and the blue team acts as the follower. Aware of the presence of excessive noise introduced by the blue team, the red team strategically counters by selecting and embedding false alternative beliefs about the misdirection pattern into the blue team’s decision-making process, aiming to subtly guide the blue team’s actions toward unconsciously revealing its true objectives. 
This formulation captures the interplay between deception and counter-deception, highlighting how adversarial agents leverage asymmetries in information and strategic adaptation to influence decision-making.

As a third contribution, we derive a semi-explicit solution for the blue team’s control problem and develop iterative and machine learning-based methods to optimize the red team’s strategy. Numerical experiments validate the proposed framework, demonstrating results consistent with theoretical intuition and highlighting the effectiveness of deception-driven strategies in adversarial interactions.

To place our work in the broader context of adversarial decision-making, we draw connections to existing frameworks while highlighting key distinctions. Adversarial interactions have also been widely modeled using partially observable stochastic games (POSGs)\cite{horak2019solving,ma2024sub,liu2022sample}. A common approach to solving POSGs involves partially observable Markov decision processes (POMDPs) or belief-space planning, in which agents maintain and update probability distributions over hidden states\cite{kurniawati2009sarsop,roy2005finding,kim2019pomhdp}. However, these methods are often computationally intensive due to the need for real-time belief updates, making them impractical for many adversarial applications. Additionally, POSGs typically assume a zero-sum structure between agents and focus heavily on numerical methods, often lacking rigorous mathematical guarantees. Another approach is the antagonistic control framework~\cite{lipp2016antagonistic}, which takes the adversary's perspective and considers attacking a control system by maximizing the cost. However, this framework is mostly for deterministic environments and requires convex constraints on the state-action space to ensure well-posedness. 
Given these challenges, we take an alternative approach by employing SHT, which offers a more direct and computationally efficient framework for modeling strategic deception and inference in adversarial settings.
Related likelihood ratio-based detection approaches in adversarial environments have also been explored in  \cite{das2024almost} for the security of cyber–physical systems.

The rest of the paper is structured as follows. Section~\ref{sec:model} presents the linear-quadratic model incorporating SHT to characterize the blue team's strategic misdirection. Section~\ref{sec:game} formulates the Stackelberg game framework and details the red team’s counter-deception strategy. Section~\ref{sec:results} provides numerical results illustrating the dynamics of the proposed game. Finally, Section~\ref{sec:discussion} discusses broader implications and potential directions for future research.

\section{The Deception Model}\label{sec:model}

This section presents a linear-quadratic stochastic control framework that captures the trade-off between achieving the blue team's primary objective and concealing its true intentions from an adversarial red team. Section~\ref{sec:IIA} formulates the blue team’s primary task as a baseline control problem, laying the foundation for incorporating deception. 
In Section~\ref{sec:IIB}, we introduce the setting of partial observability and model the red team's inference process using SHT. To facilitate this analysis, we provide necessary mathematical backgrounds and derive the likelihood ratio statistic through a technical lemma. 
Section~\ref{sec:IIC} then integrates both objectives into a unified control framework, formulating a deception-aware strategy within a linear-quadratic setting. By applying the dynamic programming principle and employing a quadratic ansatz, we derive a system of ordinary differential equations (ODEs) that provide a semi-explicit solution to the control problem. 
Furthermore, we establish the well-posedness of the ODE system, ensuring the existence and uniqueness of a global solution, which offers valuable insights into the underlying adversarial dynamics.

\subsection{The Primary Task: A Baseline Model}\label{sec:IIA}

Let \((\Omega,\F,\{\F_t\}_{t\geq 0},\Prob)\) be a filtered probability space that supports two independent standard Brownian motions, \(\{B_t\}\) and \(\{W_t\}\), with the natural filtration \(\F_t = \sigma(B_s,W_s, \forall s\in[0,t])\) generated by these processes. The blue team controls state processes \(\{V_t\}\) and \(\{Y_t\}\), which evolve under the influence of Brownian noise and the control inputs \(\{\alpha_t\}\) and \(\{\beta_t\}\), governed by the following stochastic differential equations:
\begin{align}
    \label{eqn:velocity}
    \ud V_t &= \alpha_t\ud t + \sigma_B\ud B_t,\\
    \label{eqn:position}
    \ud Y_t &= (V_t + \beta_t)\ud t + \sigma_W\ud W_t.
\end{align}
Here, the volatility parameters \(\sigma_B, \sigma_W\) are strictly positive, and the initial conditions \(V_0, Y_0 \in L^2(\Omega)\) are square-integrable random variables. 
Without loss of generality, the state and control processes are assumed to take values in \(\R\).

The primary task is defined over a finite time horizon \([0,T]\). The blue team selects control processes \(\{\alpha_t\}\) and \(\{\beta_t\}\) from the admissible set $ \A := \{\alpha:\alpha$  is progressively  measurable w.r.t.  $\{\F_t\},   \E\int_0^T|\alpha_t|^2\ud t<\infty\}$,
with the objective of minimizing the expected cost:
\begin{equation}
    J^{\text{primary}}(\alpha,\beta) := \E\Big[\int_0^T r(t,V_t,Y_t,\alpha_t,\beta_t)\ud t + g(V_T,Y_T)\Big].
    \label{eqn:primary_cost}
\end{equation}

Denoting the state variables by \(v, y \in \R\) and the control variables by \(\alpha, \beta \in \R\), the running and terminal cost functionals take the form:
\begin{align}
    \label{eqn:running_cost}
    r(t,v,y,\alpha,\beta) &= \frac{r_\alpha}{2}\alpha^2 + \frac{r_\beta}{2}\beta^2 + \frac{r_v}{2}(v-\overline{v}(t))^2,\\
    \label{eqn:terminal_cost}
    g(v,y) &= \frac{t_v}{2}(v-\overline{v}_T)^2.
\end{align}

The parameters \(r_\alpha\), \(r_\beta\), \(r_v\), and \(t_v\) are strictly positive, while \(\overline{v}_T\) is a given real value. Additionally, the target \(\overline{v}:[0,T] \to \R\) is given and assumed to be continuous.

The blue team’s decision-making follows a Markovian control framework \cite{pham2009continuous}, implying that its control processes take feedback forms:
\begin{equation}\label{eq:feedback}
    \alpha_t = \phi^\alpha(t,V_t,Y_t), \quad \beta_t = \phi^\beta(t,V_t,Y_t),
\end{equation}
where the feedback functions \(\phi^\alpha,\phi^\beta\) belong to the function class \(\Phi:= \{\phi:[0,T]\times \R\times\R\to \R:\phi \text{ is Borel measurable},\\\sup_{(t,v,y)\in [0,T]\times \R\times\R}\frac{|\phi(t,v,y)|}{1+|v|+|y|}<\infty\}\). 
This formulation ensures consistency with the constraints imposed by the admissible set \(\A\). %in \eqref{eqn:action}.

\begin{remark}[Model Interpretation]\label{rem:model_interp}
    The blue team's primary task, formulated in \eqref{eqn:velocity}--\eqref{eqn:terminal_cost}, is structured as follows. The velocity dynamics $V_t$ are governed by \eqref{eqn:velocity}, where the acceleration control \(\alpha_t\) directly influences its evolution. The position dynamics $Y_t$, described by \eqref{eqn:position}, incorporate the additional effect of \(\beta_t\), which instantaneously modifies the velocity (on the top of \(V_t\)). 
    This control \(\beta_t\) plays a key role in the strategic misdirection framework introduced later. The cost functionals \eqref{eqn:running_cost}--\eqref{eqn:terminal_cost} capture the blue team’s primary objectives: maintaining $V_t$ near the prescribed reference trajectory \(\BAR{v}(t)\) while ensuring that the terminal velocity reaches the target \(\BAR{v}_T\) at time $T$, all while minimizing the control effort, quantified as \(\frac{r_\alpha}{2}\alpha^2 + \frac{r_\beta}{2}\beta^2\).
    
    As demonstrated in Corollary~\ref{cor:baseline}, the optimal solution for the primary cost functional $J^{\text{primary}}$ results in \(\hat{\beta} \equiv 0\). This implies that, when optimally completing the primary task (i.e., misdirecting the adversary is not taken into account),
    %without the presence of the adversary, 
    the blue team has no incentive to alter its velocity instantaneously, which provides key motivation that underpins the hypothesis formulation \eqref{eqn:hypotheses} for $H_0$ in Section~\ref{sec:IIB}.
\end{remark}

\subsection{SHT-Based Intention Inference under Partial Observability}\label{sec:IIB}

As the red team, a potential adversary, seeks to infer its intentions, the blue team engages in deceptive actions that involve \(\beta_t\) in~\eqref{eqn:position}. 
In the following context, we outline the partially observable nature of the environment, the specific information available to the red team, and how these factors contribute to defining a secondary task that ultimately determines the optimal \(\hat{\beta}_t\) for the blue team.
It is important to clarify that, \textit{the red team considered in Section~\ref{sec:model} is not the actual red team, but rather the one perceived by the blue team---representing its belief in the red team’s knowledge and inference process.} The actual red team’s accessible information set will be quantified later in Section~\ref{sec:game}.

From the perspective of the blue team, the red team possesses complete knowledge of the state dynamics~\eqref{eqn:velocity}--\eqref{eqn:position}, yet lacks direct information about the blue team's primary task~\eqref{eqn:running_cost}--\eqref{eqn:terminal_cost}. While the red team can observe the sample paths of \(\{Y_t\}\), it cannot identify which underlying sample \(\omega\in\Omega\) corresponds to the observed trajectory of \(\{Y_t\}\). Consequently, the sample paths of \(\{B_t,W_t,V_t\}\) remain unobserved by the red team, resulting in a partially observable environment.

Expecting that the red team attempts to infer its intentions based on limited observations, the blue team constructs its own belief regarding the red team’s inference process, which is modeled using SHT. 
To predict the red team’s possible conclusions, the blue team places itself in the red team’s position, formulating a secondary task associated with SHT. It then strategically chooses $\beta_t$ to reduce the probability of a disadvantageous inference by the red team.

To represent the blue team's belief about the red team's inference, we focus on linear controls, motivated by the LQ structure of the primary task.

\begin{assu}[Control]
    \label{assu:linear_ctrl}
    Assume the feedback functions in~\eqref{eq:feedback} are linear in the state variables:
    \begin{align}
        \phi^\alpha(t, v, y) &= b_\alpha(t)v + c_\alpha(t)y + d_\alpha(t),\\
        \phi^\beta(t, v, y) &= b_\beta(t)v + c_\beta(t) y + d_\beta(t),
    \end{align}
    where the coefficients belong to \(C_T:= C([0,T];\R)\), the space of continuous real-valued functions on $[0,T]$.
\end{assu}

From the blue team's viewpoint, the red team conducts SHT based on the following null and alternative hypotheses:
\begin{equation}
    \label{eqn:hypotheses}
    \begin{cases}
    H_0: b_\beta \equiv 0, c_\beta \equiv 0, d_\beta \equiv 0, \\
    H_1: b_\beta \equiv 0, c_\beta = f_c, d_\beta = f_d,
    \end{cases}
\end{equation}
where \(f_c, f_d\) are functions in $C_T$. 
By strategically selecting the misdirection patterns \(f_c\) and \(f_d\), the blue team can manipulate the red team’s inference process, leading it to adopt a desired perception described by \(H_1\). 
If \(H_0\) is rejected, it indicates that the blue team is statistically more likely to behave according to \(H_1\), thus engaging in strategic misdirection.

\begin{remark}
\label{rem:hyp_interp}
The hypotheses in \eqref{eqn:hypotheses} are chosen to reflect the modeling interpretation of misdirection while preserving tractability.
Since the red team does not observe the sample paths of $\{V_t\}$, we restrict attention to misdirection that acts through the observable channel and set $b_\beta\equiv 0$.
The null hypothesis $H_0$ corresponds to the baseline case of no misdirection: Corollary~\ref{cor:baseline} shows that the blue team completes its primary task optimally with $\hat\beta\equiv 0$, so any nonzero $\beta$ is interpreted as an intentional deviation introduced to mislead the red team, and we therefore formulate the test on the feedback coefficients that define $\phi^\beta$.
We adopt a simple (non-composite) alternative in \eqref{eqn:hypotheses} because it preserves the linear--quadratic structure of the blue team's problem, under which the optimal control remains linear and the likelihood ratio admits an explicit characterization.

Under either hypothesis, $\{V_t,Y_t\}$ are well defined on $[0,T]$ with finite second moments and continuous sample paths.

Extensions beyond \eqref{eqn:hypotheses}, such as nonlinear dependencies or inferences on $\phi^\alpha$, are important but technically challenging in the absence of a linear–quadratic structure and are therefore left for future work.

\end{remark}

From a statistical perspective, test~\eqref{eqn:hypotheses} is a simple {\it vs.} simple test. The parameter space of the test is identified as \((C_T)^{3}\), where \(H_0\) and \(H_1\) correspond to the single-element subsets \(\Theta_0:= \{(0,0,0)\}\) and \(\Theta_1:= \{(0,f_c,f_d)\}\) respectively.
Given this structure, the sequential probability ratio test (SPRT)~\cite{tartakovsky2014sequential,goodman2007adaptive,schonbrodt2017sequential} naturally emerges as a suitable choice for SHT. SPRT is known to be the most powerful sequential test, as indicated by Neyman–Pearson-type results~\cite{wald1948optimum}. Throughout the subsequent discussion, SHT specifically refers to SPRT.

To construct the test statistic, we first fix some notations. Consider a measurable space \((C_T,\mathscr{B}_T)\), where \(\mathscr{B}_T\) denotes the associated $\sigma$-field. Let \(\{\eta_t,\xi_t\}\) be two stochastic processes with almost surely continuous sample paths, defined on the filtered probability space \((\Omega,\F,\{\F_t\}_{t\in[0,T]},\mathbb{P})\). The law of \(\{\eta_t\}\) is given by the probability measure \(\mu_\eta\) on \((C_T,\mathscr{B}_T)\), defined as \(\mu_\eta(B) := \Prob(\eta\in B)\) for all \(B\in\mathscr{B}_T\). If \(\mu_\eta\) is absolutely continuous with respect to \(\mu_\xi\), the Radon-Nikodym derivative \(\frac{\ud \mu_\eta}{\ud \mu_\xi}:C_T\to\R_+\) exists and represents the (likelihood ratio) statistic used in SHT.

The following technical lemma from \cite{liptser2013statistics} guarantees the existence and provides the representation of the likelihood ratio between two diffusion-type stochastic processes. 

%\hz{Fixed some typos in the lemma below. Originally, the coefficients are functions of the paths, now they are Markovian.}
\begin{lemma}[{\cite[Section~7.6.4 \& Theorem~7.19]{liptser2013statistics}}]
    \label{lem:LR}
Let \(\{\xi_t, \eta_t\}\) be two stochastic processes in \(\R^n\) with dynamics:
    \begin{align}
        \label{eqn:xi}
        \ud \xi_t &= A_t(\xi_t)\ud t + b_t(\xi_t)\ud W^m_t,\\
        \label{eqn:eta}
        \ud \eta_t &= a_t(\eta_t)\ud t + b_t(\eta_t)\ud W^m_t,
    \end{align}
    where \(A_t:\R^n\to\R^n\), \(a_t:\R^n\to\R^n\), \(b_t:\R^n\to\R^{n\times m}\) for any \(t\in[0,T]\), and \(\{W^m_t\}\) is an \(m\)-dimensional Brownian motion.\
    Denote \(x^\dagger := x^{-1}\) if \(x\neq 0\) and zero otherwise.
    If the following conditions hold:
    \begin{enumerate}
        \item \(\eta_0 = \xi_0\text{ and } \Prob(\|\eta_0\|<\infty) = 1\).
        \item \(a_t\) and \(b_t\) satisfy the standard conditions for the existence and uniqueness of the strong solution to~\eqref{eqn:eta}.
        \item For both \(\{x_t\} = \{\xi_t\}\) and \(\{x_t\} = \{\eta_t\}\),
        \begin{multline}
        \int_0^T A_t\transpose(x_t)[b_t(x_t)b_t\transpose(x_t)]^\dagger A_t(x_t)\ud t \\
        +\int_0^Ta_t\transpose(x_t)[b_t(x_t)b_t\transpose(x_t)]^\dagger a_t(x_t)\ud t<\infty,\ a.s.
        \end{multline}
        \item The equation \(b_t(x)\alpha_t(x) = A_t(x) - a_t(x)\) admits a solution for \(\alpha_t(x)\) for any \(t\in[0,T]\), \(x\in \R^n\).
    \end{enumerate}
    Then \(\mu_\xi\) is equivalent to \(\mu_\eta\) and
    \begin{multline}
        \frac{\ud \mu_\eta}{\ud \mu_\xi}(\xi) = \mathrm{exp}\Big\{-\int_0^T[A_t(\xi) - a_t(\xi)]\transpose[b_t(\xi)b_t\transpose(\xi)]^\dagger\ud \xi_t \\
        + \frac{1}{2}\int_0^T[A_t(\xi) - a_t(\xi)]\transpose[b_t(\xi)b_t\transpose(\xi)]^\dagger[A_t(\xi) + a_t(\xi)]\ud t\Big\}.
    \end{multline}
\end{lemma}

Due to space limitations, the standard Lipschitz continuity and growth conditions required for the existence and uniqueness of a strong solution to~\eqref{eqn:eta} are omitted here; full details can be found in equations~(4.110) and~(4.111) of \cite{liptser2013statistics}. As an application of Lemma~\ref{lem:LR}, the following Proposition~\ref{prop:LR_calc} provides the likelihood ratio calculation for SHT, the proof of which is detailed in Appendix~\ref{app:LR_calc}.

\begin{prop}
    \label{prop:LR_calc}
    Let \(\mu^{H_0}_{(V,Y)}\) and \(\mu^{H_1}_{(V,Y)}\) denote the law of \(\{V_t,Y_t\}\) under \(H_0\) and \(H_1\), respectively.
    Under Assumption~\ref{assu:linear_ctrl}, the likelihood ratio statistic for SHT is given by:
    \begin{align}
        \label{eqn:L_t}
        L_T(V,Y) &:= \frac{\ud \mu^{H_1}_{(V,Y)}}{\ud \mu^{H_0}_{(V,Y)}}(V,Y)\notag\\
        &= \mathrm{exp}\Bigg\{\frac{1}{\sigma_W^2}\Bigg[
        \int_0^T \left(f_c(t)Y_t + f_d(t)\right)\ud Y_t\notag\\
        &- \int_0^T V_t\left(f_c(t)Y_t + f_d(t)\right)\ud t \notag\\
        &- \frac{1}{2}\int_0^T \left(f_c(t) Y_t + f_d(t)\right)^2 \ud t    \Bigg]\Bigg\},
    \end{align}
    where the stochastic processes \((V,Y)\) evolve according to the dynamics~\eqref{eqn:velocity}--\eqref{eqn:position} under \(H_0\). 
\end{prop}

The SHT statistic \(L_T(V,Y)\) is measurable with respect to the $\sigma$-field \(\sigma(V_t,Y_t,\forall t\in[0,T])\). 
While it relies on \(V\), which remains unobservable to the red team, it is accessible to the blue team, who conducts strategic misdirection based on its belief in the red team.
A higher value of \(L_T\) increases the tendency to reject \(H_0\), whereas a lower value favors its acceptance.

In hypothesis testing, test statistics are computed under \(H_0\). However, it is important to note that when conducting the test using empirical observations \(\tilde{v},\tilde{y} \in C_T\) of the sample paths of \(\{V_t,Y_t\}\), \(L_T\) should be evaluated as \(L_T(\tilde{v},\tilde{y})\).

\subsection{Linear-Quadratic Model for Strategic Misdirection}\label{sec:IIC}

Having established the SHT (likelihood ratio) statistic \(L_T\), we now formulate the blue team’s bi-objective optimization problem by incorporating the primary task from Section~\ref{sec:IIA}. We then derive a semi-explicit solution for the optimal strategy, reducing the problem to solving a system of ODEs, and prove the global existence and uniqueness of the solution.

From the blue team’s perspective, the envisioned red team applies SHT to detect perturbations in the sample paths of \(\{Y_t\}\) in an effort to infer the blue team’s true intentions. Anticipating this, the blue team strategically introduces perturbations in \(Y_t\), even if it compromises full optimization of the primary task. Given the interpretation of \(L_T\) and Proposition~\ref{prop:LR_calc}, the blue team aims to maximize \(\log L_T\), which motivates adding the term \(-\E \log L_T\) to its primary cost~\eqref{eqn:primary_cost}, as calculated in Proposition~\ref{prop:E_log_L} (see Appendix~\ref{app:LR_calc} for the proof).

\begin{prop}
    \label{prop:E_log_L}
    Under Assumption~\ref{assu:linear_ctrl}, the expected log-likelihood ratio statistic, evaluated at the empirically observed trajectories \((V,Y)\), is given by:
    \begin{multline}
        \label{eqn:E_log_Lt}
        \E \log L_T = \frac{1}{\sigma_W^2}\E\int_0^T \Big[(f_c(t)Y_t  + f_d(t))\beta_t \\- \frac{1}{2}(f_c(t)Y_t +f_d(t))^2 \Big]\ud t.
    \end{multline} 
\end{prop}

% \begin{proof}
% Let \(Z_t := \int_0^t (f_c(s)Y_s + f_d(s))\ud W_s\). The process \(\{Z_t\}\) is a martingale with zero mean as \(\E \QV{Z}{Z}_T<\infty\). Applying the logarithm and expectation to both sides of \eqref{eqn:L_t} and substituting the dynamics from~\eqref{eqn:position} completes the proof.
% \end{proof}

Considering both the primary task~\eqref{eqn:primary_cost} and strategic misdirection~\eqref{eqn:E_log_Lt}, the blue team now aims to minimize:
\begin{equation}\label{def:J}
    J_{\text{blue}}(\alpha,\beta) := J^{\text{primary}}(\alpha, \beta) - \lambda\E\log L_T,
\end{equation}
where \(0\leq \lambda\leq r_\beta\sigma_W^2\) characterizes the intensity of strategic misdirection. The upper bound of \(\lambda\) is required in Theorem~\ref{thm:global_well} to ensure the well-posedness of the problem.

The blue team's linear-quadratic model for strategic misdirection emerges from this new stochastic control problem.
It preserves the state dynamics~\eqref{eqn:velocity}--\eqref{eqn:position} and the terminal cost~\eqref{eqn:terminal_cost}, but modifies the running cost (c.f.~\eqref{eqn:running_cost}) to:
\begin{multline}
    \label{eqn:mod_running_cost}
    h(t,v,y,\alpha,\beta) := r(t,v,y,\alpha,\beta)\\
    - \frac{\lambda}{\sigma^2_W}(f_c(t)y + f_d(t)) \beta + \frac{\lambda}{2\sigma_W^2}(f_c(t)y + f_d(t))^2,
\end{multline}
which incorporates the integrands from \eqref{eqn:E_log_Lt}.

The rest of Section~\ref{sec:IIC} is dedicated to solving the linear-quadratic control problem~\eqref{def:J}. Let \(\mathcal{V}(t, v, y):[0,T]\times \R\times \R\to\R\) be the value function, which represents the minimized cost when the system starts at $(V_t, Y_t) = (v, y)$. 
Applying the dynamic programming principle, $\mathcal{V}$ satisfies the Hamilton-Jacobi-Bellman (HJB) equation:
\begin{multline}
    \partial_t \mathcal{V} + \inf_{\alpha,\beta}\{\alpha\partial_v \mathcal{V} + (v+\beta)\partial_y \mathcal{V} + \frac{1}{2}\sigma_B^2\partial_{vv}\mathcal{V} \\+ \frac{1}{2}\sigma_W^2\partial_{yy}\mathcal{V}  
    + h(t,v,y,\alpha,\beta)\} = 0,
    \label{eqn:HJB}
\end{multline}
with a terminal condition $
    \mathcal{V}(T,v,y) = \frac{t_v}{2}\big(v - \overline{v}_T\big)^2$.
Minimizing over $\alpha,\beta$ produces the optimal controls: 
\begin{align}
    \label{eqn:alpha_hat}
    \hat{\alpha} &=  -\frac{1}{r_\alpha}\partial_v \mathcal{V},\\
    \label{eqn:beta_hat}
    \hat{\beta} &= \frac{1}{r_\beta}\Big[ \frac{\lambda}{\sigma_W^2}\Big(f_c(t)y + f_d(t)\Big) - \partial_y \mathcal{V}\Big].
\end{align}
Adopting a quadratic ansatz, we assume
\begin{equation}
    \mathcal{V}(t,v,y) = \frac{\mu_t}{2}v^2 + \eta_tvy + \frac{\rho_t}{2}y^2 + \gamma_t v + \theta_t y +  \xi_t,
\end{equation}
where \(\mu\), \(\eta\), \(\rho\), \(\gamma\), \(\theta\), \(\xi \in C_T\) .
Plugging \eqref{eqn:alpha_hat} and~\eqref{eqn:beta_hat} back into \eqref{eqn:HJB} and collecting corresponding coefficients yield a system of ODEs:
\begin{equation}
    \label{eqn:ODE_system}
    \begin{cases}
        \dot{\mu}_t = \frac{1}{r_\alpha}\mu_t^2 + \frac{1}{r_\beta}\eta_t^2 - 2\eta_t - r_v,\\
        \dot{\eta}_t = \frac{1}{r_\alpha}\mu_t \eta_t + \frac{1}{r_\beta}\rho_t \eta_t - \rho_t - \frac{\lambda}{r_\beta \sigma_W^2}\eta_t f_c(t),\\
        \dot{\rho}_t = \frac{1}{r_\alpha}\eta_t^2 + \frac{1}{r_\beta} \rho_t^2 - \frac{2\lambda}{r_\beta \sigma_W^2}\rho_t f_c(t) + (\frac{\lambda^2}{r_\beta \sigma_W^4} - \frac{\lambda}{\sigma_W^2})f_c^2(t),\\
        \dot{\gamma}_t = \frac{1}{r_\alpha} \mu_t \gamma_t + \frac{1}{r_\beta}\eta_t \theta_t - \theta_t + r_v \overline{v}(t) - \frac{\lambda}{r_\beta \sigma_W^2}\eta_t f_d(t),\\
        \dot{\theta}_t = \frac{1}{r_\alpha}\eta_t \gamma_t + \frac{1}{r_\beta}\rho_t \theta_t - \frac{\lambda}{r_\beta \sigma_W^2}\theta_t f_c(t) - \frac{\lambda}{r_\beta \sigma_W^2}f_d(t) \rho_t \\\quad\quad + (\frac{\lambda^2}{r_\beta \sigma_W^4} - \frac{\lambda}{\sigma_W^2})f_c(t) f_d(t),\\
        \dot{\xi}_t = \frac{1}{2r_\alpha}\gamma_t^2 + \frac{1}{2r_\beta}\theta_t^2 - \frac{1}{2}\sigma_B^2 \mu_t - \frac{1}{2}\sigma_W^2 \rho_t - \frac{\lambda}{r_\beta \sigma_W^2}f_d(t) \theta_t \\
        \quad \quad- \frac{r_v[\overline{v}(t)]^2}{2} + (\frac{\lambda^2}{2 r_\beta \sigma_W^4} - \frac{\lambda}{2 \sigma_W^2})f_d^2(t),
    \end{cases}
\end{equation}
with terminal conditions
\begin{align}
    \mu_T &= t_v,\ 
    \eta_T = 0, \
    \rho_T = 0, \
    \gamma_T = -t_v\overline{v}_T,\\
    \theta_T &= 0,\
    \xi_T = \frac{t_v}{2}\big(\overline{v}_T\big)^2.
\end{align}
The semi-explicit solution to this linear-quadratic control problem is given by
\begin{equation}
    \label{eqn:solution}
    \begin{aligned}
        \hat{\alpha}(t,v,y) = &-\frac{\mu_t}{r_\alpha} v - \frac{\eta_t}{r_\alpha} y - \frac{\gamma_t}{r_\alpha},\\
        \hat{\beta}(t,v,y) = &-\frac{\eta_t}{r_\beta} v + \Big(\frac{\lambda}{r_\beta\sigma_W^2}f_c(t) - \frac{\rho_t}{r_\beta}\Big) y \\ &+ \Big(\frac{\lambda}{r_\beta\sigma_W^2}f_d(t) - \frac{\theta_t}{r_\beta}\Big),
    \end{aligned}
\end{equation}
which justifies Assumption~\ref{assu:linear_ctrl}. This solution provides key insights into optimal deception strategies in adversarial settings. Although $\alpha$ and $\beta$ are designated for the primary task and strategic misdirection respectively, the optimal solution inherently couples them, leading to $\hat\alpha$ and $\hat\beta$ differing from the solutions obtained if these tasks were treated independently. By comparing the coefficients in \(\hat{\beta}\) with those in \(H_1\) of hypothesis~\eqref{eqn:hypotheses}, it becomes evident that the blue team strategically spends its misdirection efforts towards a specific pattern dictated by \(f_c\) and \(f_d\). Philosophically, the effectiveness of optimally introducing perturbations is meaningful only when constraints define the desired perturbation patterns—this serves as a key motivation for employing a simple {\it vs.} simple SHT rather than incorporating information divergences into the cost functionals. However, the dependencies of \(\hat{\alpha}\) and \(\hat{\beta}\) on \( f_c \) and \( f_d \) are highly nonlinear, as variations in $f_c$ or $f_d$ propagate through multiple state variables via the interdependent dynamics of \((\mu_t, \eta_t, \rho_t, \gamma_t, \theta_t)\). Consequently, these dependencies generate complex, coupled effects rather than straightforward additive or multiplicative control adjustments.

%\hz{Changed the following paragraphs, canceled the local well-posedness result, commented some unused references.}

To establish the well-posedness of the ODE system~\eqref{eqn:ODE_system}, we present the following Theorem~\ref{thm:global_well}, with its proof provided in our earlier work \cite{zhou2025integrating}. 

% To establish the well-posedness of the ODE system~\eqref{eqn:ODE_system}, we present only the statements of Theorems~\ref{thm:local_well} and \ref{thm:global_well}, while moving their proofs to Appendix~\ref{app:thm_well}. Denoting \(x(t) := (\mu_t, \eta_t, \rho_t, \gamma_t, \theta_t, \xi_t)\transpose\), the system~\eqref{eqn:ODE_system} can be expressed in the compact form \(\dot{x}(t) = F(t,x(t))\), where \(F:[0,T]\times \R^6\to\R^6\).

% \begin{theorem}[Local Existence and Uniqueness]
%     \label{thm:local_well}
%     There exists \(t_0\in\R_+\) and \(m>0\) such that the solution \(x(t)\) to the ODE system~\eqref{eqn:ODE_system} exists and is unique, when \(t\in[T-t_0,T]\) and \(\|x(t) - x(T)\|_\infty\leq m\).
% \end{theorem}

\begin{theorem}[Global Existence and Uniqueness]
    \label{thm:global_well}
    Under the condition \(0\leq\lambda \leq r_\beta \sigma_W^2\), the ODE system~\eqref{eqn:ODE_system} has a unique global solution on \([0,T]\), for any \(T>0\).
\end{theorem}

With global existence and uniqueness established, we next analyze the optimal strategies in scenarios where strategic misdirection is absent or SHT becomes trivial. As shown in Corollary~\ref{cor:baseline}, the optimal control $\hat\beta$ remains trivial when misdirection is not considered $(\lambda =0)$ or when the hypotheses in~\eqref{eqn:hypotheses} are identical $(H_0 = H_1)$. In such cases, the model simplifies to the blue team's primary task, where the optimal strategy does not involve any perturbations.

\begin{cor}
    \label{cor:baseline}
    Under the condition \(0\leq\lambda \leq r_\beta \sigma_W^2\), if \(f_c \equiv f_d \equiv 0\) or \(\lambda=0\), then the unique solution to the ODE system~\eqref{eqn:ODE_system} satisfies \(\eta\equiv\rho\equiv\theta\equiv 0\), implying that \(\hat{\beta} \equiv 0\).
\end{cor}

\begin{proof}
    If \(f_c\equiv f_d\equiv 0\) or \(\lambda = 0\), the ODEs governing \(\rho_t\) and \(\theta_t\) reduce to homogeneous equations, leading to \(\eta\equiv \rho\equiv \theta\equiv 0\) as a solution to system~\eqref{eqn:ODE_system}. 
    By the uniqueness result in Theorem~\ref{thm:global_well}, this solution is unique.
    Substituting these values into \eqref{eqn:solution}, we conclude that \(\hat{\beta}\equiv 0\).
\end{proof}

\begin{remark}
    In particular, when \(\lambda = r_\beta\sigma_W^2\), \(\eta\equiv \rho\equiv \theta\equiv 0\) is the unique solution to system~\eqref{eqn:ODE_system}, using an argument analogous to Corollary~\ref{cor:baseline}.
    In this case, the optimal controls in~\eqref{eqn:solution} have the forms
    \begin{equation}
        \hat{\alpha}(t,v,y) = -\frac{\mu_t}{r_\alpha} v - \frac{\gamma_t}{r_\alpha},\quad
        \hat{\beta}(t,v,y) =  f_c(t) y + f_d(t).
    \end{equation}
    Here, the optimal \(\hat{\beta}\) depends only on the position \(y\), precisely matching the form of hypothesis \(H_1\) in~\eqref{eqn:hypotheses}, while \(\hat{\alpha}\) depends solely on velocity \(v\). In other words, when the strategic misdirection is at full intensity, the roles of \(\alpha\) and \(\beta\) become fully decoupled:
     \(\alpha\) adjusts acceleration to fulfill the primary task, and \(\beta\) is entirely responsible for introducing misdirection.
\end{remark}

\section{The Stackelberg Game}\label{sec:game}

This section explores how the actual red team, rather than the envisioned one, strategically responds to the blue team's misdirection. Without loss of generality, we focus on a simplified version of the LQ model introduced in Section~\ref{sec:IIC}, where we set \(f_d\equiv \BAR{v}\equiv 0\) and \(\BAR{v}_T = 0\), resulting in $\gamma \equiv \theta \equiv 0$.

The actual red team is assumed to have knowledge of the governing dynamics~\eqref{eqn:velocity}--\eqref{eqn:position} but can only observe the sample paths of \(\{Y_t\}\). 
However, it has agents capable of:
\begin{itemize}
    \item Revealing the blue team's current misdirection choice \(f_c\), the functional forms of its strategies \(\hat{\alpha},\hat{\beta}\) in \eqref{eqn:solution}, and knowledge on how \(f_c\) affects the time-dependent coefficients in \(\hat{\alpha},\hat{\beta}\) (encoded in the ODEs~\eqref{eqn:ODE_system})\footnote{Lacking prior knowledge of the LQ structure, the red team can not uniquely determine the blue team’s cost functionals. This corresponds to the inverse reinforcement learning (IRL) problem, where a given control strategy is generally optimal under multiple cost functionals \cite{ng2000algorithms}.}.
    \item Manipulating the blue team's misdirection signal $f_c$. 
\end{itemize}
Understanding that the observed trajectories of \(\{Y_t\}\) reflect perturbations introduced by the blue team, the red team aims to lure the blue team into voluntarily reducing its misdirection efforts, while carefully avoiding any actions that might reveal the existence of its agents.

\smallskip
\noindent\textbf{Stackelberg Game Formulation.} The red team initiates the interaction by optimizing $f_c$, after which the blue team solves its LQ problem with SHT, using the instilled $\hat f_c$, as described in Section~\ref{sec:model}. This interaction forms a Stackelberg game, where the red team’s optimization consists of two key components:
\begin{itemize}
    \item \textbf{Misdirection Control}: it aims to minimize $\E \log L_T$, based on the information provided in \eqref{eqn:solution}, encouraging the blue team to act more in line with \(H_0\) rather than \(H_1\) over $[0,T]$. 
    If the red team confirms $H_0$ in \eqref{eqn:hypotheses}, it concludes that the blue team is not introducing any deliberate perturbations into the observed process \(\{Y_t\}\), in which case the realized sample path can be reliably used for further analysis.
    Notably, the red team targets the distributional impact of misdirection mitigation, rather than reacting to the pathwise fluctuations of the observed process $\{Y_t\}$.
    \item \textbf{Regularization}: To reflect the blue team's skepticism toward externally influenced \(f_c\), we introduce a penalty term \(\mathcal{P}(f_c)\) to promote more realistic strategic interaction. 
\end{itemize}

Without regularization, the optimal \(f_c\) that minimizes \(\E \log L_T\) may be unrealistic, particularly when it lies too close to the null hypothesis, making it unacceptable for the blue team to adopt in its envisioned SHT. To illustrate this concern, we first consider the unregularized case, where the red team aims to minimize \(\E \log L_T\) without accounting for the blue team's skepticism. The objective function of the red team is 
\begin{equation}\label{eqn:red_min_obj_unreg}
    J_{\text{red}}(f_c) := \E[\log \hat{L}_T],
\end{equation}
where $\hat{L}_T$ is defined in \eqref{eqn:L_t}, evaluated at the trajectories $(\hat V, \hat Y)$, which follow \eqref{eqn:velocity}--\eqref{eqn:position}  under the optimal controls $(\hat\alpha, \hat \beta)$ provided in \eqref{eqn:solution}.

Using \eqref{eqn:velocity}--\eqref{eqn:position} and \eqref{eqn:solution}, Itô's formula and Fubini's theorem imply that
\begin{multline}
\label{eqn:exp_log}
   \E[\log \hat L_T] = \frac{1}{\sigma_W^2}\int_0^T \bigg[\left(-\frac{\eta_t}{r_\beta}h_{11}(t) - \frac{\rho_t}{r_\beta}h_{02}(t)\right)f_c(t) \\
    + \left(\frac{\lambda}{r_\beta \sigma_W^2}-\frac{1}{2}\right)h_{02}(t)f_c^2(t) \bigg]\ud t,
\end{multline} 
where $h_{ij}(t) :=  \E [\hat{V}_t^i\hat{Y}_t^j]$ are the mixed second moments of $(\hat V_t, \hat Y_t)$, with \((i,j) \in \{(0,2),(1,1),(2,0)\}\).
Furthermore, the moments can be computed via the coupled ODE system:
\begin{equation}\label{eqn:ODE_moments}
    \begin{cases}
        \dot{h}_{20}(t) = -2\frac{\mu_t}{r_\alpha}h_{20}(t) - 2\frac{\eta_t}{r_\alpha}h_{11}(t) + \sigma_B^2, \\
        \dot{h}_{11}(t) = \left(\frac{\lambda}{r_\beta\sigma_W^2}f_c(t) - \frac{\rho_t}{r_\beta} - \frac{\mu_t}{r_\alpha}\right)h_{11}(t) \\
        \quad \quad\quad\quad + \left(1-\frac{\eta_t}{r_\beta}\right)h_{20}(t) - \frac{\eta_t}{r_\alpha}h_{02}(t), \\
        \dot{h}_{02}(t) = 2\left(1 - \frac{\eta_t}{r_\beta}\right)h_{11}(t) \\
        \quad \quad\quad\quad + 2\left(\frac{\lambda}{r_\beta\sigma_W^2}f_c(t) - \frac{\rho_t}{r_\beta}\right)h_{02}(t) + \sigma_W^2,
    \end{cases}
\end{equation}
with given initial conditions
%\begin{equation}
$$   h_{20}(0) = \E V_0^2,\quad h_{11}(0) = \E V_0Y_0,\quad h_{02}(0) = \E Y_0^2.$$

Intuitively, if the red team chooses $f_c\equiv 0$, then hypotheses $H_0$, $H_1$ coincide, effectively causing the blue team to voluntarily abandon strategic misdirection.
Therefore, we term the red team's control $f_c \equiv 0$ as the trivial optimizer, which is proved to be the local minimizer of the optimization problem~\eqref{eqn:red_min_obj_unreg}.
This result is presented as Theorem~\ref{thm:trival_opt}.

\begin{theorem} [Trivial Optimizer]\label{thm:trival_opt}
Denote by \(L^2([0,T];\R)\) the collection of square-integrable real-valued functions on \([0,T]\), and \(\delta A(t,\cdot)(f_c)(u)\) the first variation of \(A\) in \(f_c\) with respect to the variation \(u\in L^2([0,T];\R)\), where \(A\) is a real-valued function of \(f_c\). If the following conditions hold:
\begin{enumerate}
    \item $\frac{1}{2}r_\beta\sigma_W^2<\lambda\leq r_\beta\sigma_W^2.$
    \item \(\mu_t\) is first-order differentiable, and \(\eta_t,\rho_t,h_{11}(t),h_{02}(t)\) are second-order differentiable in \(f_c\) in the Gateaux sense, for any \(t\in[0,T]\).
    \item 
    \(\delta A(t,\cdot)(f_c\equiv 0)(u)\) is differentiable in \(t\), and furthermore, \(\frac{\ud}{\ud t}[\delta A(t,\cdot)(f_c\equiv 0)(u)] = \delta \dot{A}(t,\cdot)(f_c\equiv 0)(u)\), for any \(t\in[0,T],\ u\in L^2([0,T];\R)\).
\end{enumerate}
Then $f_c\equiv 0$ is the local minimizer of the red team's optimization problem~\eqref{eqn:red_min_obj_unreg}.
\end{theorem}

\begin{proof}
Recall that the red team minimizes the objective~\eqref{eqn:red_min_obj_unreg} with respect to \(f_c\) in its unregularized optimization problem.
As the first step, we clarify functional dependencies, pointing out that the solutions to both systems~\eqref{eqn:ODE_system} and~\eqref{eqn:ODE_moments} depend on $f_c$.
In the following context, we explicitly specify such dependence whenever necessary, e.g., $\mu_t$ can be equivalently denoted as $\mu(t, f_c)$. 

The proof utilizes tools from the calculus of variation. The main idea is to show that the first variation in \(f_c\) is constantly zero, whereas the second variation is strictly positive, when both variations are evaluated at \(f_c\equiv 0\).
Due to the differentiability in the Gateaux sense, all the first and second variations mentioned in the proof are finite.

For simplicity of the notation, define 
\begin{equation}
\label{eqn:def_G}
G(t,f_c) := \frac{\eta(t,f_c)}{r_\beta}h_{11}(t,f_c) + \frac{\rho(t,f_c)}{r_\beta}h_{02}(t,f_c).
\end{equation}
By~\eqref{eqn:red_min_obj_unreg}--\eqref{eqn:exp_log}, it follows that
\begin{multline}
    J_{\text{red}}(f_c) = \frac{1}{\sigma_W^2}\int_0^T \Bigg[-G(t,f_c) f_c(t)\\
    + \left(\frac{\lambda}{r_\beta \sigma_W^2}-\frac{1}{2}\right)h_{02}(t,f_c)f_c^2(t)\Bigg]\ud t.
\end{multline}

Calculate the first variation of \(J_{\text{red}}\) in $f_c$, with respect to the variation $u \in L^2([0,T];\R)$:
\begin{multline}
    \delta J_\red(f_c)(u)= \frac{1}{\sigma_W^2} \Bigg[-\int_0^T [\delta G(t,\cdot)(f_c)(u)]f_c(t)\ud t \\
    + \Big(\frac{\lambda}{r_\beta \sigma_W^2}-\frac{1}{2}\Big)\Big(\int_0^T [\delta h_{02}(t,\cdot)(f_c)(u)]f_c^2(t)\ud t \\
    + 2\int_0^T h_{02}(t,f_c) f_c(t)\, u(t)\ud t\Big)- \int_0^T G(t,f_c)\, u(t)\ud t \Bigg].
\end{multline}
Evaluate both sides at \(f_c\equiv 0\).
By Corollary~\ref{cor:baseline}, $G(t, f_c\equiv 0)= 0,\ \forall t\in[0,T]$, which implies $\delta J_\red (f_c\equiv 0)(u) = 0,\ \forall u\in L^2([0,T];\R)$.
Hence $f_c\equiv 0$ is a critical point of \eqref{eqn:red_min_obj_unreg}. 

Calculate the second variation of \(J_{\text{red}}\) in $f_c$, with respect to variations $ u,v \in L^2([0,T];\R)$: 
\begin{multline}
    \delta^2 J_{\text{red}}(f_c)(u,v)
    = \frac{1}{\sigma_W^2} \Bigg[-\int_0^T [\delta^2 G(t,\cdot)(f_c)(u,v)]f_c(t)\ud t \\
    - \int_0^T [\delta G(t,\cdot)(f_c)(u)]v(t)\ud t - \int_0^T [\delta G(t,\cdot)(f_c)(v)] u(t)\ud t \\
    + \Big(\frac{\lambda}{r_\beta \sigma_W^2}-\frac{1}{2}\Big)
    \Big(2\int_0^T [\delta h_{02}(t,\cdot)(f_c)(u)]f_c(t)v(t)\ud t \\
    + \int_0^T [\delta^2 h_{02}(t,\cdot)(f_c)(u,v)]f_c^2(t)\ud t \\
    +2\int_0^T h_{02}(t,f_c)v(t)u(t)\ud t \\
    + 2 \int_0^T [\delta h_{02}(t,\cdot)(f_c)(v)]f_c(t)u(t)\ud t\Big)\Bigg].
\end{multline}
Evaluating both sides at $u\equiv v$ and $f_c \equiv 0$ yields
\begin{multline}
    \label{eqn:second_var}
    \delta^2 J_{\text{red}}(f_c\equiv 0)(u,u)\\
    =\frac{2}{\sigma_W^2}\Bigg[- \int_0^T [\delta G(t,\cdot)(f_c\equiv 0)(u)]u(t)\ud t\\
    +\Big(\frac{\lambda}{r_\beta \sigma_W^2}-\frac{1}{2}\Big)\int_0^T h_{02}(t,f_c\equiv 0) u^2(t)\ud t\Bigg].
\end{multline}

On the other hand, for fixed \(t\in[0,T]\), taking the first variation in \(f_c\) on both sides of the ODE system~\eqref{eqn:ODE_system} with respect to the variation \(u\in L^2([0,T];\R)\) yields 
\begin{equation}
\label{eqn:function_ODE}
\begin{cases}
\begin{aligned}
&\delta \dot{\mu}
= \frac{2}{r_\alpha}\mu(\delta \mu) - 2(\delta \eta) + \frac{2}{r_\beta}\eta(\delta \eta) \\
\end{aligned} \\
\begin{aligned}
\delta \dot{\eta} &= \frac{1}{r_\alpha}\Bigl(\mu(\delta \eta) + \eta(\delta \mu)\Bigr) + \frac{1}{r_\beta}\Bigl(\rho(\delta \eta) + \eta(\delta \rho)\Bigr) - (\delta\rho)\\
&-\frac{\lambda}{r_\beta\sigma_W^2}\Bigl(f_c(t)\,(\delta \eta) + \eta \,u(t)\Bigr)\\
\end{aligned} \\
\begin{aligned}
\delta \dot{\rho}
&= \frac{2}{r_\alpha}\eta(\delta \eta)  + \frac{2}{r_\beta}\rho(\delta \rho) -\frac{2\lambda}{r_\beta\sigma_W^2}\Bigl(f_c(t)\,(\delta \rho) + \rho \,u(t)\Bigr) \\
& - 2\Bigl(\frac{\lambda}{\sigma_W^2} - \frac{\lambda^2}{r_\beta\sigma_W^4}\Bigr)f_c(t)u(t)
\end{aligned}
\end{cases},
\end{equation}
where the shorthand notations \( A := A(t,f_c)\) and \(\delta A := \delta A(t,\cdot)(f_c)(u)\) are adopted for objects \(A\in\{\mu,\eta,\rho\}\).
For given \(f_c\) and \(u\), \eqref{eqn:function_ODE} is a system of function-valued ODEs, with terminal conditions \(\delta \mu(T,\cdot)(f_c)(u)  =  \delta \eta(T,\cdot)(f_c)(u) = \delta \rho(T,\cdot)(f_c)(u) = 0\).

Evaluate all the equations at \(f_c\equiv 0\).
By Corollary~\ref{cor:baseline}, \(\eta(t,f_c\equiv 0)= \rho(t,f_c\equiv 0) = 0,\ \forall t\in[0,T]\), which implies 
\begin{equation}
        \begin{cases}
            \begin{aligned}
            \delta \dot{\mu}|_{f_c\equiv 0}
            = \tfrac{2}{r_\alpha}\mu(t,f_c\equiv 0)(\delta \mu)|_{f_c\equiv 0} - 2(\delta \eta)|_{f_c\equiv 0}
            \end{aligned}\\
            \begin{aligned}
            \delta \dot{\eta}|_{f_c\equiv 0} &= \tfrac{1}{r_\alpha}\mu(t,f_c\equiv 0)(\delta \eta)|_{f_c\equiv 0} -(\delta \rho)|_{f_c\equiv 0}
            \end{aligned}\\
            \begin{aligned}
            &\delta \dot{\rho}|_{f_c\equiv 0} \equiv 0
            \end{aligned}
        \end{cases}.
    \end{equation}
Interchanging the differentiation in \(t\) and the first variation in \(f_c\) yields \((\delta \eta)|_{f_c\equiv 0} \equiv (\delta \rho)|_{f_c\equiv 0} \equiv 0,\ \forall u\in L^2([0,T];\R),t\in[0,T]\).
By the definition~\eqref{eqn:def_G}, we get
\begin{equation}
    \label{eqn:G_f_c}
    \delta G(t,\cdot)(f_c\equiv 0)(u) = 0, \ \forall u\in L^2([0,T];\R),t\in[0,T].
\end{equation}
Combining~\eqref{eqn:second_var} and~\eqref{eqn:G_f_c} yields \(\delta^2 J_{\text{red}}(f_c\equiv 0)(u,u)> 0\), \(\forall u \in L^2([0,T];\R),u\not\equiv 0 \ \mathrm{a.e.}\), which concludes the proof.
\end{proof}

\smallskip

While the result for the trivial optimizer $f_c \equiv 0$ is mathematically consistent with the motivation of the red team in the Stackelberg game, it is unrealistic for the blue team to actually adopt it, since \(f_c\equiv 0\) does not induce any misdirection, which runs against the blue team's motivation.
Consequently, it is necessary to add a penalty term that models the blue team's level of skepticism towards the manipulated $f_c$.
With this regularization effect introduced, the red team's objective is formulated as:
    \begin{equation}\label{def:Obj-red}
        J_{\text{red}}(f_c) := \E[\log \hat{L}_T] + \frac{\lambda_{\text{reg}}}{\sigma_W^2} \mathcal{P}(f_c),
    \end{equation}
where $\lambda_{\text{reg}}>0$ characterizes the regularization intensity. 

To model the penalty term $\mathcal{P}(f_c)$, we introduce a trust region (proximity) constraint on $f_c$.
Initially, the blue team uses $f^{\text{initial}}_c$.
After the manipulation of the red team's agents, the blue team adopts a new misdirection pattern induced by \(f_c\) but simultaneously increases its skepticism, which is proportional to the difference between $f^{\text{initial}}_c$ and $f_c$. The trust region constraint thus penalizes deviations from $f^{\text{initial}}_c$ to $f_c$, ensuring a controlled and credible shift. Mathematically, we choose
\begin{multline}\label{def:reg}
    \mathcal{P}[f_c] = \int_0^T (f_c(t)- f_c^{\text{initial}}(t))^2 \ud t \\
    \text{ or } -\int_0^T f_c^{\text{initial}}(t)\log \frac{f_c(t)}{f_c^{\text{initial}}(t)} \ud t.
\end{multline}
The first is a quadratic penalty, and the second, inspired by KL-divergence, treats both \(f_c^{\text{initial}}\) and \(f_c\) as unnormalized densities on $[0,T]$.

Since $ J_{\text{red}}(f_c)$ does not admit a closed-form minimizer, Section~\ref{sec:IV-B} explores three numerical approaches to compute the optimal $f_c$ and presents the numerical results.

%%%%%%%%%%%%%%%%%%%%%%%%%%%%%%%%%%%%%%%%%%%%%%%%%%%%%%%%%%%%%%%%%%%%%%%%%%%%%%%%

\section{Numerical Experiments}\label{sec:results}

In this section, we present numerical algorithms and experiments on deceptive and counter-deceptive strategies. Specifically, we focus on the blue team's optimal misdirection strategies $\hat\beta_t$ (discussed in Section~\ref{sec:model}) and the red team's optimal misdirection pattern $\hat f_c$ (discussed in Section~\ref{sec:game}). 
Finally, by combining both teams' solutions, we present the multiple-round interaction within the Stackelberg game, showing how the red team gradually weakens the misdirection effect of the blue team through the manipulation of \(f_c\).

\subsection{Blue Team's Optimal Control $\hat \beta_t$}\label{sec:IV-A}

Focusing on the blue team's control problem outlined in Section~\ref{sec:IIC}, we plot trajectories of the blue team's velocity $\hat{V}_t$, observed position $\hat{Y}_t$, and optimal controls $\hat{\alpha}_t$, $\hat{\beta}_t$.
All subplots in Figures~\ref{fig:blue_test1}--\ref{fig:blue_test2} share the set of model parameters:
\begin{center}
    $T = 1, \sigma_B = \sigma_W = 0.25, V_0 = 2, Y_0 = 4, r_\alpha = 1,r_\beta = 10$\\
    $r_v = 1,t_v = 1,\BAR{v}_T = 1,\BAR{v}(t) = 2 - t,f_d\equiv 0.$
\end{center}

\begin{figure}[!htbp]
    \centering
    \includegraphics[width=\linewidth]{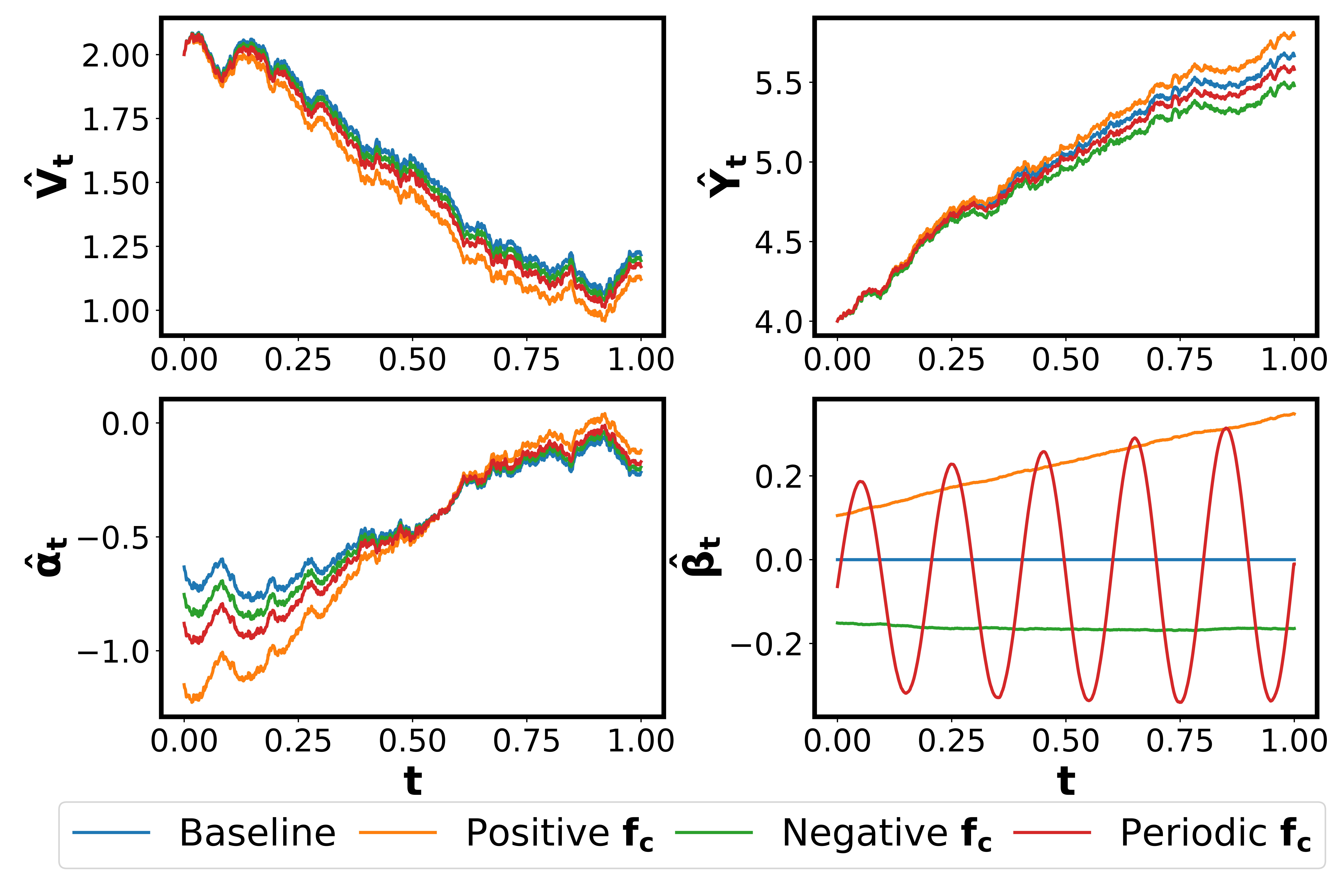}
    \caption{Comparisons of the optimal trajectories \eqref{eqn:velocity}--\eqref{eqn:position} and controls \eqref{eqn:solution} with \(\lambda = 0.075\) across different choices of \(f_c\): baseline \(f_c\equiv 0\), positive \(f_c\equiv 0.5\), negative \(f_c\equiv -0.25\), and periodic \(f_c(t) = 0.5\sin(10\pi t)\).
    }\label{fig:blue_test1}
\end{figure}

\begin{figure}[!htbp]
    \centering
    \includegraphics[width=\linewidth]{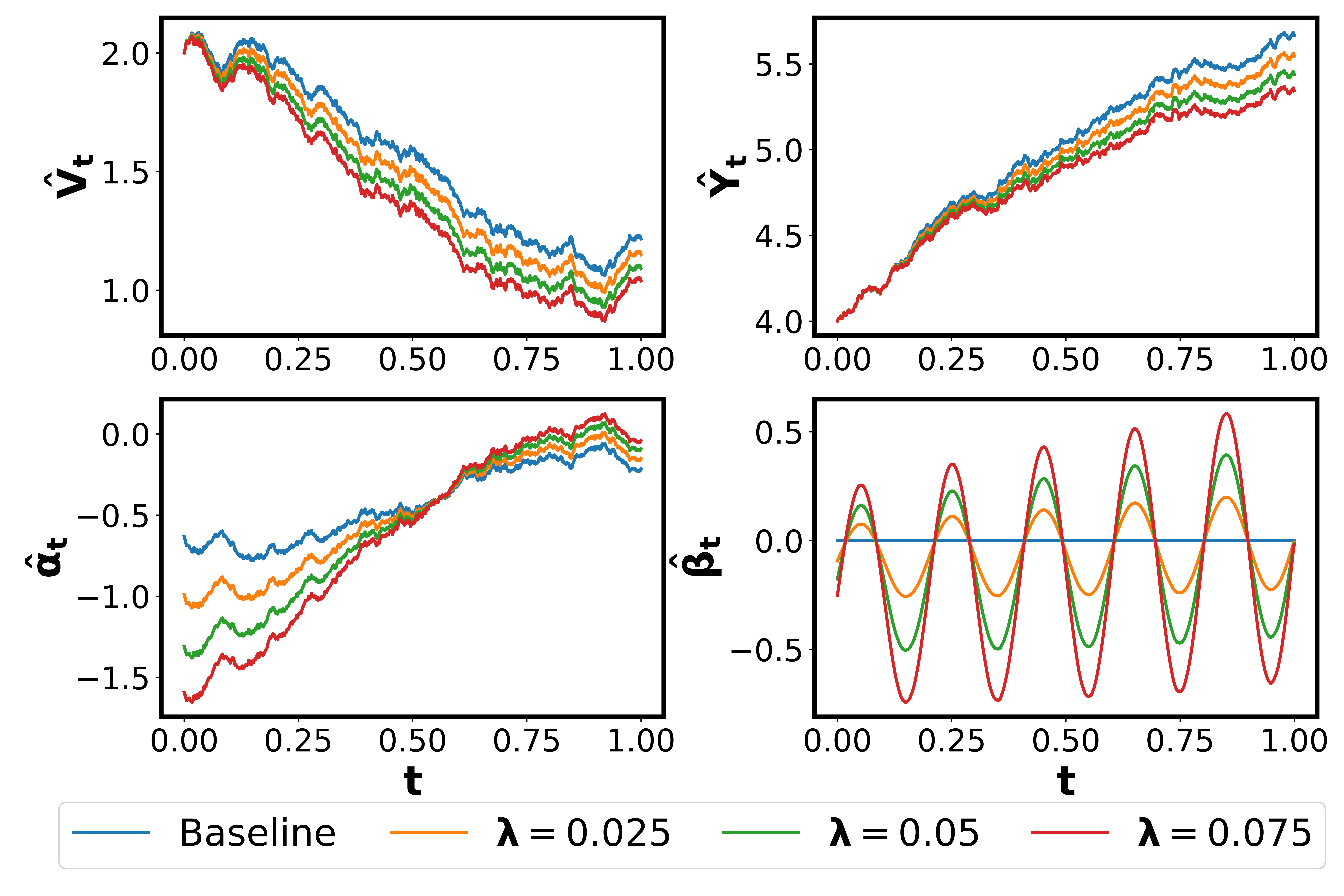}
    \caption{Comparisons of the optimal trajectories \eqref{eqn:velocity}--\eqref{eqn:position} and controls \eqref{eqn:solution} with \(f_c(t) = \sin(10\pi t)\) across different values of \(\lambda\).}
    \label{fig:blue_test2}
\end{figure}

Figure~\ref{fig:blue_test1} illustrates how different choices of \( f_c \) affect the blue team’s optimal trajectories. The baseline case \( f_c \equiv 0 \) represents a scenario where no intentional perturbation is introduced (c.f. Corollary~\ref{cor:baseline}).
A constant positive \(f_c\) perturbs the position trajectory upward, while a constant negative \(f_c\) perturbs the position trajectory downward.
The periodic \( f_c\), on the other hand, introduces oscillatory deviations, which could potentially obscure the true intention of the blue team in a dynamic environment.

Figure~\ref{fig:blue_test2} examines the impact of varying \(\lambda\), the intensity of strategic misdirection, while fixing \( f_c \) to be periodic. As \(\lambda\) increases, the blue team places greater emphasis on misleading the red team, resulting in more pronounced deviations from the baseline trajectory. We use Monte Carlo simulations with 10,000 paths to quantify the trade-off: $J^{\text{primary}}(\hat\alpha,\hat\beta)$ increases from $0.33$ at $\lambda=0$ to $0.44$ at $\lambda = 0.025$, $0.78$ at $\lambda = 0.05$ and $1.32$ at $\lambda=0.075$, while $\E[\log \hat L_T]$ improves from $-96.98$ to $-87.26, -78.14, -69.55$.
This highlights the trade-off between control effort and deception effectiveness: higher \(\lambda\) values lead to more aggressive strategic misdirection, at the cost of fulfilling the primary task in a less efficient way.

\subsection{Red Team's Optimal Control $\hat f_c$}\label{sec:IV-B}

To compute the red team's regularized optimal control $\hat f_c$ in \eqref{def:Obj-red}, we examine three algorithms: a fixed point iteration (FPI), a neural network-based method  (NN), and the forward-backward sweep method (FBS). Using these methods, we present experiments on the optimal control for both regularization choices in~\eqref{def:reg}.

FPI starts with an initial guess $f_c^{(0)}$. In the $i$-th iteration, the systems  \eqref{eqn:ODE_system} and \eqref{eqn:ODE_moments} are solved with $f_c$ replaced by $f_c^{(i)}$. The optimization in \eqref{def:Obj-red} is then performed using the obtained $(\mu, \eta, \rho)$ and $(h_{20}, h_{11}, h_{02})$, yielding $f_c^{(i+1)}$. The process iterates until convergence. We outline the procedure in Algorithm \ref{alg:FPI}, while leaving the technical details in the dependence of \(f_c^{(i+1)}\) on \(f_c^{(i)}\) to Appendix~\ref{app:alg}.

\begin{algorithm}
\caption{Fixed Point Iteration (FPI)}\label{alg:FPI}
\KwIn{Initial guess $f_c^{(0)}$}

\(f_c = f_c^{(0)}\)\;
\While{not converge}{
    Solve \eqref{eqn:ODE_system}, \eqref{eqn:ODE_moments} for $(\mu_t, \eta_t, \rho_t)$ and $(h_{20}, h_{11}, h_{02})$\;
    Update \(f_c\) as the minimizer of~\eqref{def:Obj-red} given the obtained $(\mu_t, \eta_t, \rho_t)$ and $(h_{20}, h_{11}, h_{02})$\;
}

\KwOut{The optimizer $f_c$}
\end{algorithm}

In the NN method, we directly parameterize $f_c$ using a feedforward neural network that takes $t\in [0,T]$ as input and outputs a real number, as inspired by \cite{han2016deep}.
Within each training epoch, we discretize the time horizon and simulate the ODE systems \eqref{eqn:ODE_system} and \eqref{eqn:ODE_moments} using the Euler scheme.
This enables the computation of $J_{\text{red}}$, which is set as the loss for updating neural network parameters. 
The complete procedure is outlined in Algorithm~\ref{alg:NN_method}.

\begin{algorithm}
\caption{Neural Network Method (NN)}\label{alg:NN_method}
\KwIn{Initial network parameters $\theta^{(0)}$}

Initialize network parameters $\theta =  \theta^{(0)}$\;
\While{not converge}{
    Simulate~\eqref{eqn:ODE_system} and~\eqref{eqn:ODE_moments} using Euler schemes\;

    Compute $J_{\text{red}}$ based on simulated trajectories\;

    Update $\theta$ with loss  $J_{\text{red}}$\;
}
\KwOut{A trained neural network that parameterizes the optimal \(f_c\)}
\end{algorithm}

The FBS method \cite[Ch.~4]{lenhart2007optimal} treats the red team's optimization problem as a deterministic optimal control problem.
Leveraging Pontryagin’s maximum principle \cite{mangasarian1966sufficient}, one derives the optimal \(f_c\) by minimizing the Hamiltonian, which is further characterized by the solutions to a system of forward-backward ODEs (FBODEs).
FBS then serves as a specific technique that solves the FBODEs, by alternating between solving the state equations and adjoint equations. 
In particular, convergence guarantees of FBS are established when the coefficients of the FBODEs satisfy certain boundedness and Lipschitz continuity conditions \cite{mcAsey2012convergence}. We outline the procedure in Algorithm~\ref{alg:FBS} and leave the derivations of the FBODEs to Appendix~\ref{app:alg}.

\begin{algorithm}
\caption{Forward-Backward Sweep (FBS)}
\label{alg:FBS}
\DontPrintSemicolon
\KwIn{Initial guess $f_c^{(0)}$}
$f_c = f_c^{(0)}$\;
\While{not converge}{
  Solve the state equations with the given $f_c$\;
  Solve the adjoint equations with the given $f_c$ and the solutions to the state equations\;
  Update \(f_c\) as the minimizer of the Hamiltonian\;
}
\Return The optimizer $f_c$\;
\end{algorithm}

Figure~\ref{fig:redtest9&11} presents the optimal $\hat f_c$ computed using different algorithms and penalty terms $\mathcal{P}(f_c)$. The following parameter values are used \footnote{Such choices align with the simplified LQ model, in which we have set \(f_d\equiv \BAR{v}\equiv 0\) and \(\BAR{v}_T = 0\).}:
\begin{center}
$T = 0.1, \sigma_B = \sigma_W = 0.1, V_0 = 1, Y_0 = 2, r_\alpha = 1,r_\beta = 10$\\
    $r_v = 1,t_v = 1,\BAR{v}_T = 0,\BAR{v}(t) \equiv 0,f_d\equiv 0.$
\end{center}

\begin{figure}[htbp]
    \centering
    \includegraphics[width=1.0\linewidth]{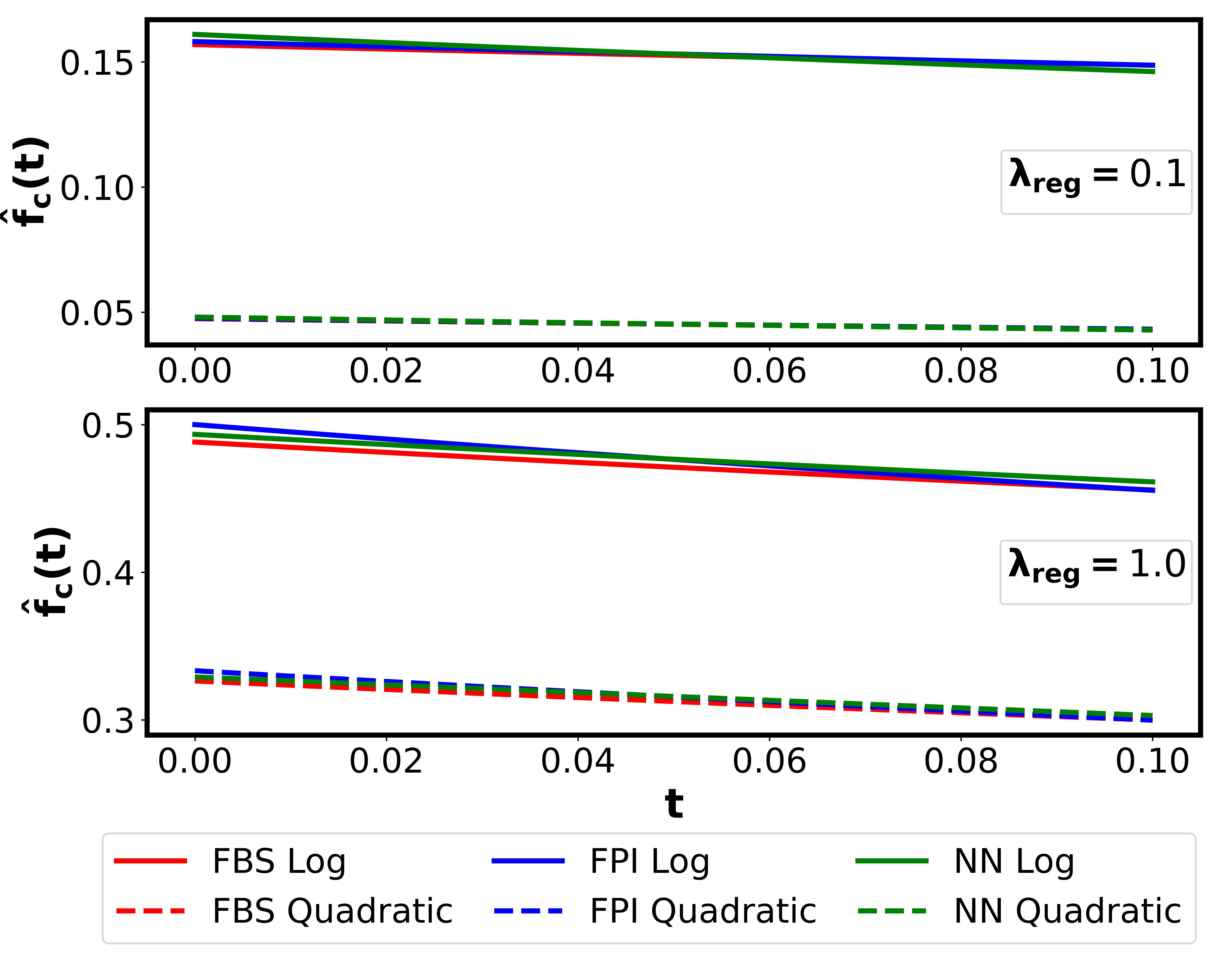}
    \caption{Plots of the optimal $\hat f_c$ across different algorithms, penalties, and values of \(\lambda_{\text{reg}}\).}
    \label{fig:redtest9&11}
\end{figure}

For both the penalty intensities $\lambda_{\text{reg}}$ and the choices of $\mathcal{P}(f_c)$, all three methods (FPI, NN and FBS) produce largely consistent results. 
A key observation is that increasing $\lambda_{\text{reg}}$ shifts $\hat f_c$ closer to $f_c^{\text{initial}} \equiv 1$ but also increases the value of $\E[\log \hat L_T]$, affirming the trade-off between counter-deception and avoiding skepticism. Additionally, the logarithmic penalty consistently yields results closer to 1, aligning with its steeper decay near zero.

Numerically, $\E[\log \hat L_T]$ under $\hat f_c$ is sensitive to both the penalty and $\lambda_{\text{reg}}$. With the logarithmic penalty, values increase from about $0.50$ at $\lambda_{\text{reg}}=0.1$ to $5.00$ at $\lambda_{\text{reg}}=1.0$, while the quadratic penalty yields smaller increases (from $0.04$ to $2.17$). For reference, the unregularized baseline $f_c^{\text{initial}}\equiv1$ gives $\E[\log \hat L_T]=23.21$, regardless of penalty. These results confirm that by tuning $f_c$, the red team can counteract deception while controlling skepticism, thereby shaping the blue team’s misdirection strategies.

\vspace{1.5em}

\subsection{Red-blue Interaction in Multiple Rounds}\label{sec:IV-C}

Combining optimal controls in Sections~\ref{sec:IV-A}--\ref{sec:IV-B} facilitates a complete understanding of the red-blue interaction within the Stackelberg game that occurs for multiple rounds.

Under the simplified version of the LQ model (cf. Section~\ref{sec:game}), the blue team starts with a given pattern \(f_c\) within each round, which induces the optimal controls \(\hat{\alpha},\hat{\beta}\).
Setting \(f_c^{\text{initial}}\) as the \(f_c\) currently adopted by the blue team, the red team calculates and instills \(\hat{f}_c\), which serves as the blue team's misdirection pattern in the next round.

\begin{figure}[htbp]
    \centering
    \includegraphics[width=1.0\linewidth]{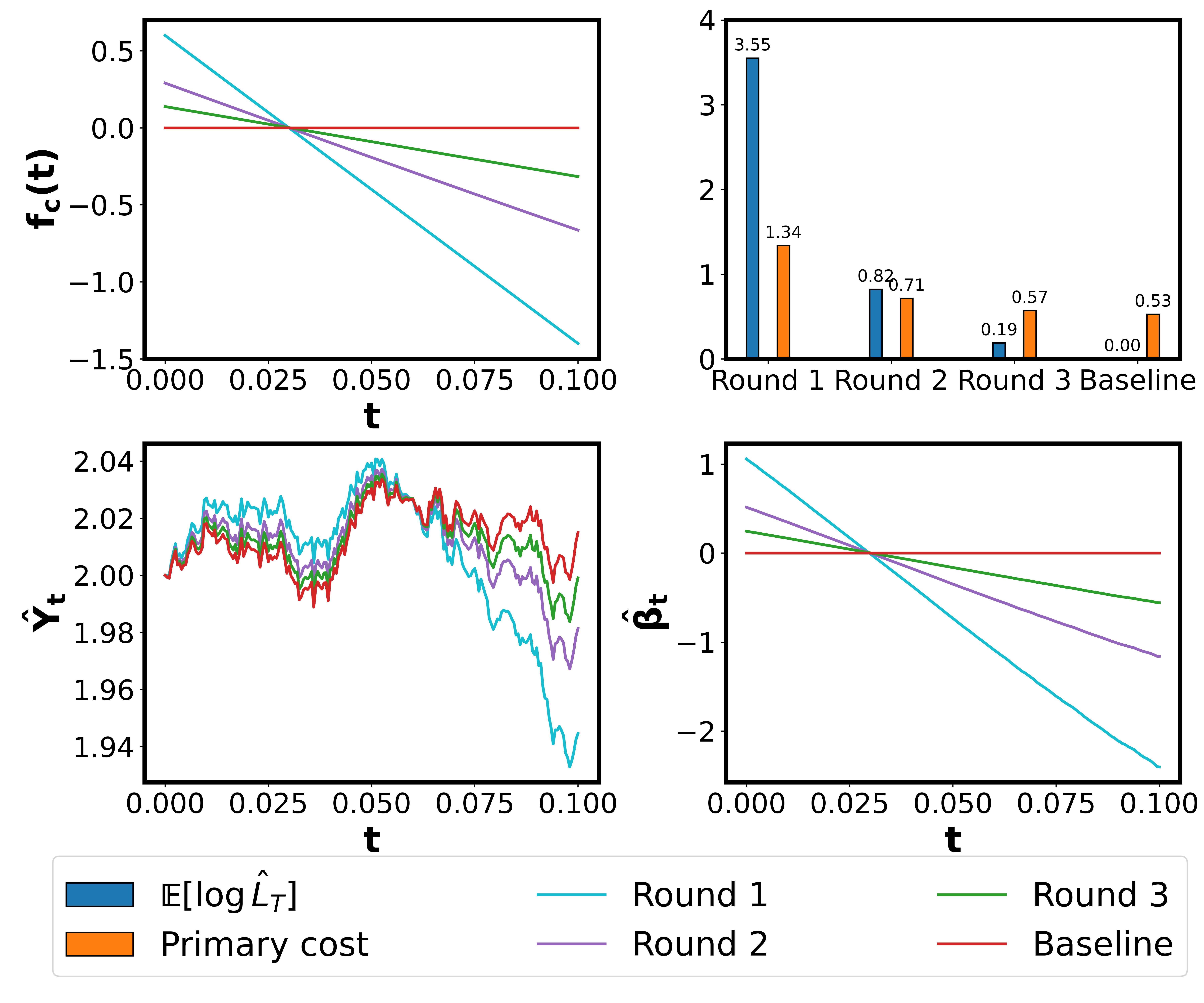}
    \caption{Comparisons of \(f_c\), \(J^{\text{primary}}(\hat{\alpha},\hat{\beta})\), \(\E[\log\hat{L}_T]\), optimal state trajectories~\eqref{eqn:position} and controls~\eqref{eqn:solution} across multiple rounds of red-blue interaction within the Stackelberg game.}
    \label{fig:red_blue}
\end{figure}

Figure~\ref{fig:red_blue} demonstrates the choice of \(f_c\), \(J^{\text{primary}}(\hat{\alpha},\hat{\beta})\), \(\E[\log\hat{L}_T]\), and the blue team's state and control trajectories across multiple rounds of the Stackelberg game.
The baseline case stands for \( f_c \equiv 0 \), which aligns with the trivial optimizer in Theorem~\ref{thm:trival_opt}.
As a proof of principle, the red team adopts the NN solver with a quadratic penalty, as validated in Section~\ref{sec:IV-B}.
The following parameter values are used:
\begin{center}
$T = 0.1, \sigma_B = \sigma_W = 0.15, V_0 = 1, Y_0 = 2, r_\alpha = 2,r_\beta = 10$\\
    $r_v = 1,t_v = 1,\BAR{v}_T = 0,\BAR{v}(t) \equiv 0,f_d\equiv 0$\\
    $\lambda=0.2,\lambda_{\text{reg}} = 1.5,f_c(t) = 20(0.03 - t).$
\end{center}

As the game progresses, the blue team's optimal state and control trajectories gradually approach the baseline case, and the misdirection pattern \(f_c\) gets closer to the trivial optimizer.
In round \(1\), the trajectories exhibit strong misdirection, reflected by a large \(\E[\log\hat{L}_T]\), but this comes at the expense of a higher primary cost \(J^{\text{primary}}(\hat{\alpha},\hat{\beta})\).
After 2 rounds,  both metrics are close to their baseline values, illustrating the effectiveness of the red team’s counter-deception measures.

One may wonder whether repeated iterations would drive $f_c$ to zero; we note that such successive reductions depend on the blue team’s willingness to accept the red team’s manipulations in each round, which may not always hold in practice. Developing models that explicitly capture this acceptance dynamic is a worthwhile direction for future research.

\section{Conclusion}\label{sec:discussion}

This paper presents an expanded study of deception and counter-deception in partially observable Stackelberg games, expanding our earlier conference paper~\cite{zhou2025integrating}.
Within a linear-quadratic framework, we model the red team's strategic manipulation and the blue team's response through optimal control under sequential hypothesis testing. 
Compared to \cite{zhou2025integrating}, this paper: (i) establishes a new theoretical result (Theorem~\ref{thm:trival_opt}) that motivates the introduction of regularization in the red team's optimization (cf. Section~\ref{sec:game}); 
(ii) provides detailed derivations of test statistics, numerical algorithms (Appendix~\ref{app:alg}), and hyperparameter settings (see Appendix~\ref{app:hyper});
(iii) illustrates multi-round red–blue interactions through additional numerical experiments (Figure~\ref{fig:red_blue}). 

In contrast to the active inverse learning frameworks of~\cite{ward2023active, kim2025deceptive}, which study how a leader can influence an observer or follower to facilitate accurate inference of latent preferences or reward functions, our model emphasizes how strategic deception and belief shaping can emerge. 

Future research directions include multi-dimensional extensions, where controls, velocities and positions take values in \(\mathbb{R}^d\).
As quantities in the hypotheses become matrix- or vector-valued functions, the likelihood ratio statistic can be computed using multi-dimensional variants of Lemma~\ref{lem:LR} \cite{liptser2013statistics}, without introducing qualitative changes to the structure of the problem.
Other extensions of our current model include more realistic trust region penalties, nonlinear dynamics and hypotheses, multi-agent reinforcement learning for adaptive deception and detection, robustness against evolving adversaries, and large-scale interactions in mean-field games with common noise. 

Although our analysis is presented in a two-player Stackelberg game setting, it also serves as a building block for control over networked dynamical systems.
In such settings, the state may describe a collection of interconnected subsystems coupled through a graph (e.g., consensus or mean-field type interactions), with actions applied locally at selected nodes. The partial observability considered here naturally extends to networked settings, where measurements are available only at a subset of nodes or through aggregated network outputs. From this perspective, our Stackelberg formulation captures adversarial inference and strategic signaling for a networked plant, while a systematic treatment of network architectures and scalable implementations is left for future work.

These extensions aim to deepen the theoretical foundations of deception-aware control and broaden its applicability to complex networked systems.

%\newpage
%\pagebreak

% Appendices
\appendices

\section{Proofs of Propositions~\ref{prop:LR_calc}--\ref{prop:E_log_L}}
\label{app:LR_calc}

\begin{proof}[Proof of Proposition~\ref{prop:LR_calc}]

Using the notations of Lemma~\ref{lem:LR}, identify \(m = n = 2\), \(\xi_t\) as \((V_t,Y_t)\) under \(H_0\), and \(\eta_t\) as \((V_t,Y_t)\) under \(H_1\).
From dynamics~\eqref{eqn:velocity}--\eqref{eqn:position}, it is clear that
\begin{align}
    &A_t(v,y) = \begin{bmatrix}
        \phi^\alpha(t,v,y) \\ v
    \end{bmatrix}, \
    a_t(v,y) = \begin{bmatrix}
        \phi^\alpha(t,v,y) \\ v + \phi^\beta(t,v,y)
    \end{bmatrix},\\
    &\qquad \qquad \qquad b_t(v,y) = \begin{bmatrix}
        \sigma_B & 0\\
        0 & \sigma_W
    \end{bmatrix}.
\end{align}
Under Assumption~\ref{assu:linear_ctrl}, \(\xi\) and \(\eta\) have linear dynamics with the same initial condition.
Besides, \(b_t\) is always invertible.
Therefore, all conditions in Lemma~\ref{lem:LR} hold.
The likelihood ratio~\eqref{eqn:L_t} follows from a direct substitution.
\end{proof}

\begin{proof}[Proof of Proposition~\ref{prop:E_log_L}]
    Define \(Z_t := \int_0^t (f_c(s)Y_s + f_d(s))\ud W_s\). Since \(\E \QV{Z}{Z}_T<\infty\), it follows that \(\{Z_t\}\) is a martingale with zero mean.
    Taking logarithm and expectation on both sides of \eqref{eqn:L_t} yield
    \begin{align}
        &\E \log L_T = \frac{1}{\sigma_W^2}\E\Bigg[
        \int_0^T \left(f_c(t)Y_t + f_d(t)\right)\ud Y_t\notag\\
        &- \int_0^T V_t\left(f_c(t)Y_t + f_d(t)\right)\ud t \notag - \frac{1}{2}\int_0^T \left(f_c(t) Y_t + f_d(t)\right)^2 \ud t    \Bigg].
    \end{align}
    Substituting the dynamics~\eqref{eqn:position} concludes the proof.
\end{proof}

\section{Detailed Derivations of the Algorithms FPI and FBS in Section~\ref{sec:IV-B}}\label{app:alg}

\subsection{Detailed Derivations of FPI}

Within the $i$-th iteration of FPI, the systems~\eqref{eqn:ODE_system} and~\eqref{eqn:ODE_moments} are numerically solved with $f_c$ replaced by $f_c^{(i)}$, yielding $(\mu^i, \eta^i, \rho^i)$ and $(h^i_{20}, h^i_{11}, h^i_{02})$, where the superscript \(i\) implies the dependence on \(f_c^{(i)}\), i.e., \(\mu^i_t := \mu(t,f_c^{(i)})\).

In the case of a quadratic penalty \(\mathcal{P}[f_c] = \int_0^T (f_c(t)-1)^2 \ud t\), the red team's objective~\eqref{def:Obj-red} has the form
\begin{multline}
    \label{eqn:quad_FPI}
   J_{\red}(f_c) = \frac{1}{\sigma_W^2}\int_0^T \bigg[\left(-\frac{\eta_t}{r_\beta}h_{11}(t) - \frac{\rho_t}{r_\beta}h_{02}(t)\right)f_c(t) \\
    + \left(\frac{\lambda}{r_\beta \sigma_W^2}-\frac{1}{2}\right)h_{02}(t)f_c^2(t) +\lambda_{\REG}(f_c(t)-1)^2 \bigg]\ud t.
\end{multline}
Minimizing the integrand in~\eqref{eqn:quad_FPI} with respect to \(f_c\) pointwisely yields the FPI update \(f_c^{(i+1)} = \tfrac{2 \lambda_\REG r_\beta \sigma_W^2 + \sigma_W^2 \eta^i h_{11}^i + \sigma_W^2 \rho^i h_{02}^i}{2 \lambda_\REG r_\beta \sigma_W^2 - \left(r_\beta \sigma_W^2 - 2 \lambda\right) h_{02}^i}\).

For the logarithmic penalty \(\mathcal{P}[f_c] = -\int_0^T \log f_c(t) \ud t\), the red team's objective~\eqref{def:Obj-red} has the form
\begin{multline}
    \label{eqn:log_FPI}
   J_{\red}(f_c) = \frac{1}{\sigma_W^2}\int_0^T \bigg[\left(-\frac{\eta_t}{r_\beta}h_{11}(t) - \frac{\rho_t}{r_\beta}h_{02}(t)\right)f_c(t) \\
    + \left(\frac{\lambda}{r_\beta \sigma_W^2}-\frac{1}{2}\right)h_{02}(t)f_c^2(t) -\lambda_{\REG}\log f_c(t) \bigg]\ud t.
\end{multline}
Minimizing the integrand in~\eqref{eqn:log_FPI} with respect to \(f_c\) pointwisely yields the FPI update \(f_c^{(i+1)} = \tfrac{\tfrac{1}{r_\beta}(\eta^ih_{11}^i + \rho^ih_{02}^i) + \sqrt{\Delta^i}}{4(\tfrac{\lambda}{r_\beta \sigma_W^2}-\frac{1}{2})h_{02}^i}\), where
\begin{equation}
    \Delta^i := \tfrac{1}{r_\beta^2}(\eta h_{11}^i + \rho^i h_{02})^2 + 8\lambda_\REG \left(\tfrac{\lambda}{r_\beta \sigma_W^2}-\tfrac{1}{2}\right)h_{02}\notag.
\end{equation}

\subsection{Detailed Derivations of FBS}

In the case of a quadratic penalty \(\mathcal{P}[f_c] = \int_0^T (f_c(t)-1)^2 \ud t\), the red team hopes to optimize its control \(f_c\) that minimizes the expected cost~\eqref{eqn:quad_FPI}, subject to the state dynamics~\eqref{eqn:ODE_system} and~\eqref{eqn:ODE_moments}.
Therefore, the red team's optimization can be identified as a deterministic optimal control problem with zero terminal cost.
For the convenience of notations, we denote by \(x :=  (\mu, \eta, \rho, h_{20}, h_{11}, h_{02})\) the state process of this optimal control problem, which follows the dynamics (cf.~\eqref{eqn:ODE_system} and~\eqref{eqn:ODE_moments}):
\begin{equation}\label{eqn:ODE_fbstate}
    \begin{cases}
    \dot{x_1} = \frac{1}{r_\alpha} x_1^2 + \frac{1}{r_\beta} x_2^2 - 2x_2 - r_v\\
    \dot{x_2} = \frac{1}{r_\alpha} x_1x_2  + \frac{1}{r_\beta}x_2x_3 - x_3 - \frac{\lambda}{r_\beta \sigma_W^2}f_cx_2\\
    \dot{x_3} = \frac{1}{r_\alpha} x_2^2 + \frac{1}{r_\beta} x_3^2 - \frac{2\lambda}{r_\beta \sigma_W^2} f_c x_3 - ( \frac{\lambda}{\sigma_W^2} - \frac{\lambda^2}{r_\beta \sigma_W^4})f_c^2 \\
    \dot{x_4} = -\frac{2}{r_\alpha} x_1 x_4 - \frac{2}{r_\alpha} x_2 x_5 + \sigma_B^2\\
    \dot{x_5} = (1-\frac{x_2}{r_\beta})x_4 + ( \frac{\lambda}{r_\beta \sigma_W^2}f_c - \frac{x_3}{r_\beta} - \frac{x_1}{r_\alpha})x_5 - \frac{x_2}{r_\alpha}x_6\\
    \dot{x_6} = 2(1 - \frac{x_2}{r_\beta})x_5 + 2( \frac{\lambda}{r_\beta \sigma_W^2} f_c - \frac{x_3}{r_\beta})x_6 + \sigma_W^2
    \end{cases},
\end{equation}
with given initial and terminal conditions 
\begin{equation}
\begin{aligned}
    &x_1(T) = t_v,           &\; &x_2(T) = 0,              &\; &x_3(T) = 0, \\
    &x_4(0) = \E V_0^2,      &\; &x_5(0) = \E V_0 Y_0,     &\; &x_6(0) = \E Y_0^2.
\end{aligned}
\end{equation}

We specify the Hamiltonian for this control problem:  
\begin{multline}
    H(x, f_c, \psi) = \sum_{i=1}^6\psi_i \dot{x_i} + \left(-\tfrac{x_2}{r_\beta}x_5 - \tfrac{x_3}{r_\beta}x_6\right)f_c \\
    + \left(\tfrac{\lambda}{r_\beta \sigma_W^2}-\tfrac{1}{2}\right)x_6f_c^2
    + \lambda_\REG (f_c-1)^2,
\end{multline}
where $\psi:[0,T]\to\R^6$ denotes dual variables.
By Pontryagin's maximum principle, the dual variables satisfy the following system of adjoint equations:
\begin{equation}\label{eqn:ODE_fbcostate}
    \begin{cases}
    \dot{\psi}_1 = -\psi_1\frac{2 x_{1}}{r_{\alpha}} - \psi_2\frac{x_{2}}{r_{\alpha}} + \psi_4\frac{2 x_{4}}{r_{\alpha}} + \psi_5\frac{x_{5}}{r_{\alpha}}\\
    \dot{\psi}_2 = \frac{f_cx_{5}}{r_{\beta}}- 2\psi_{1}(\frac{x_{2}}{r_{\beta}} - 1) -\psi_{2}(\frac{x_{1}}{r_{\alpha}} + \frac{x_{3}}{r_{\beta}}  - \frac{f_c\lambda}{r_{\beta}\sigma_{W}^{2}})  \\
    \quad \quad - \psi_3\frac{2 x_{2}}{r_{\alpha}}+ \psi_{4}\frac{2 x_{5}}{r_{\alpha}} + \psi_{5}\Big(\frac{x_{4}}{r_{\beta}} + \frac{x_{6}}{r_{\alpha}}\Big)  +\psi_{6}\frac{2 x_{5}}{r_{\beta}}\\
    \dot{\psi}_3=\frac{f_c x_{6}}{r_{\beta}}  -\psi_{2}(\frac{x_{2}}{r_{\beta}} - 1) - 2\psi_{3}(\frac{x_{3}}{r_{\beta}} - \frac{f_c\lambda}{r_{\beta}\,\sigma_{W}^{2}}) + \psi_{5}\frac{ x_{5}}{r_{\beta}} \\
    \quad\quad + \psi_{6}\frac{2 x_{6}}{r_{\beta}}\\
    \dot{\psi}_4= \psi_{4}\frac{2 x_{1}}{r_{\alpha}} + \psi_{5}(\frac{x_{2}}{r_{\beta}} - 1)\\
    \dot{\psi}_5= \frac{f_c x_{2}}{r_{\beta}}+\psi_{4}\frac{2 x_{2}}{r_{\alpha}} + 
    \psi_5(-\frac{\lambda}{r_\beta\sigma_W^2}f_c+\frac{x_3}{r_\beta} +\frac{x_1}{r_\alpha})\\
    \quad\quad+ 2\psi_{6}(\frac{x_{2}}{r_{\beta}} - 1)\\
    \dot{\psi}_6= \frac{f_c x_{3}}{r_{\beta}} -\frac{1}{2}\,f_c^{2}(\frac{2\lambda}{r_{\beta}\sigma_{W}^{2}}- 1)+\psi_{5}\frac{ x_{2}}{r_{\alpha}} + 2\psi_{6}(\frac{x_{3}}{r_{\beta}} - \frac{f_c \lambda}{r_{\beta} \sigma_{W}^{2}})
    \end{cases}
\end{equation}
with given terminal conditions \(\psi_i(T) = 0,\ \forall 1\leq i\leq 6\).
The control update follows from the minimizer of the Hamiltonian:
\begin{multline}
    \hat{f_c} = \tfrac{2 \lambda_\REG r_\beta \sigma_W^4 + \lambda \psi_2 \sigma_W^2 x_2 + 2\lambda \psi_3 \sigma_W^2 x_3}{2\lambda_\REG r_\beta \sigma_W^4 - 2\lambda \psi_3 r_\beta \sigma_W^2 + 2\lambda^2 \psi_3 - \left(r_\beta \sigma_W^4 - 2\lambda \sigma_W^2\right) x_6} \\
    + \tfrac{\left(\sigma_W^4 x_2 - \lambda \psi_5 \sigma_W^2\right) x_5 + \left(\sigma_W^4 x_3 - 2\lambda\psi_6 \sigma_W^2\right) x_6}{2\lambda_\REG r_\beta \sigma_W^4 - 2\lambda \psi_3 r_\beta \sigma_W^2 + 2\lambda^2 \psi_3 - \left(r_\beta \sigma_W^4 - 2\lambda \sigma_W^2\right) x_6}.
\end{multline}

For the logarithmic penalty, the state dynamics~\eqref{eqn:ODE_fbstate} and the adjoint equations~\eqref{eqn:ODE_fbcostate} remain the same, while the only difference lies in control update \(\hat{f_c} = \tfrac{-B + \sqrt{B^2 + 4A\lambda_{\REG}}}{2A}\), where
\begin{align}\label{eqn:dHdf_fbsweep}
    A &:= x_6 (\tfrac{2\lambda}{r_\beta \sigma_W^2} - 1) - 2(\tfrac{\lambda}{\sigma_W^2} - \tfrac{\lambda^2}{r_\beta \sigma_W^4})\psi_3,\\
    B &:= -\tfrac{x_2x_5 + x_3x_6}{r_\beta} + \tfrac{\lambda \psi_5 x_5 + 2\lambda \psi_6x_6 - \lambda \psi_2x_2 - 2 \lambda x_3 \psi_3}{r_\beta \sigma_W^2}.
\end{align}

\section{Hyperparameters for Numerical Results in Section~\ref{sec:results}}\label{app:hyper}

In Section~\ref{sec:results}, we discretize the time horizon \([0,T]\) into \(N_T\) subintervals of equal lengths \(h = T/N_T\), and denote the time discretization scheme by \(\Delta := \{kh:0\leq k\leq N_T,k\in\mathbb{N}\}\), as the collection of the endpoints of all subintervals.
Section~\ref{sec:IV-A} uses \(N_T = 1000\) while Section~\ref{sec:IV-B} uses \(N_T = 200\).
All the ODEs are numerically solved by using the explicit Runge-Kutta method of order 8 \cite{wanner1996solving}, except for the ODEs within the NN method, where gradients of neural network parameters need to be tracked.

In Section~\ref{sec:IV-B}, \(f_c\) is numerically maintained as a vector \(f_c(\Delta)\in\R^{N_T + 1}\) evaluated at all the endpoints within \(\Delta\).
For the FPI and FBS solvers, iterations are terminated when the first time \(\|f_c^{(i+1)}(\Delta) - f_c^{(i)}(\Delta)\|_2<10^{-3}\) is satisfied.

For the NN solver, we use a four-layer feedforward neural network with the hyperbolic tangent activation function, which has \(32\) hidden neurons within each of the hidden layer.
It is worth noting that, when the logarithmic penalty is adopted, \(f_c\) must be strictly positive, which motivates us to add an extra layer of component-wise exponential output activation function after the last affine layer.
For parameter updates of neural networks, we adopt the Adam optimizer with an initial learning rate \(\eta = 0.001\).
Altogether \(N_{\mathrm{epoch}} = 500\) training epochs are carried out.


\begin{thebibliography}{99}
    \bibitem{zi2007art} 
    Sun Zi, {\it The Art of War: Sun Zi's Military Methods}, Columbia University Press, 2007.
    
    \bibitem{aggarwal2016cyber}
    P. Aggarwal, C. Gonzalez and V. Dutt, Cyber-Security: Role of Deception in Cyber-Attack Detection, {\it Adv. Hum. Factors Cybersecurity, Proc. AHFE Int. Conf. Hum. Factors Cybersecurity}, July 27-31, 2016, Florida, USA, pp 85-96.

    \bibitem{arkin2011moral}
    R. C. Arkin, P. Ulam and A. R. Wagner, Moral Decision Making in Autonomous Systems: Enforcement, Moral Emotions, Dignity, Trust, and Deception, {\it Proc. IEEE}, vol. 100, 2011, pp 571-589.
    
    \bibitem{gerschlager2005deception}
    C. Gerschlager, {\it Deception in Markets: An Economic Analysis}, Springer, 2005.

    \bibitem{back2000imperfect}
    K. Back, C. Cao and G. Willard, Imperfect Competition among Informed Traders, {\it J. Finance}, vol. 55, 2000, pp 2117-2155.

    
    \bibitem{yager2008knowledge}
    R. R. Yager, A Knowledge-Based Approach to Adversarial Decision Making, {\it Int. J. Intell. Syst.}, vol. 23, 2008, pp 1-21.
    
    \bibitem{rajendran2011blue}
    J. Rajendran, V. Jyothi and R. Karri, Blue Team Red Team Approach to Hardware Trust Assessment, {\it Proc. IEEE Int. Conf. Comput. Des. (ICCD)}, 2011, pp 285-288.
    
    \bibitem{liptser2013statistics}
    R. S. Liptser and A. N. Shiryaev, {\it Statistics of Random Processes: I. General Theory}, Springer Science \& Business Media, 2013.

    \bibitem{tartakovsky2014sequential}
    A. Tartakovsky, I. Nikiforov and M. Basseville, {\it Sequential Analysis: Hypothesis Testing and Changepoint Detection}, CRC Press, 2014.
    
    \bibitem{goodman2007adaptive}
    N. A. Goodman, P. R. Venkata and M. A. Neifeld, Adaptive Waveform Design and Sequential Hypothesis Testing for Target Recognition with Active Sensors, {\it IEEE J. Sel. Top. Signal Process.}, vol. 1, 2007, pp 105-113.
    
    \bibitem{schonbrodt2017sequential}
    F. Schönbrodt, E. Wagenmakers, M. Zehetleitner and M. Perugini, Sequential Hypothesis Testing with Bayes Factors: Efficiently Testing Mean Differences, {\it Psychol. Methods}, vol. 22, 2017, pp 322.
    
    \bibitem{pham2009continuous}
    H. Pham, {\it Continuous-Time Stochastic Control and Optimization with Financial Applications}, Springer Science \& Business Media, 2009.
    
    \bibitem{wald1948optimum}
    A. Wald and J. Wolfowitz, Optimum Character of the Sequential Probability Ratio Test, {\it Ann. Math. Stat.}, 1948, pp 326-339.

    \bibitem{bain2009fundamentals}
    A. Bain and D. Crisan, {\it Fundamentals of Stochastic Filtering}, Springer, 2009.
    
    \bibitem{davis1977linear}
    M. Davis, {\it Linear Estimation and Stochastic Control}, Chapman, 1977.
    
    
    \bibitem{horak2019solving}
    K. Horák and B. Bošanskỳ, Solving Partially Observable Stochastic Games with Public Observations, {\it Proc. AAAI Conf. Artif. Intell.}, vol. 33, 2019, pp 2029-2036.
    
    \bibitem{ma2024sub}
    O. Ma, Y. Pu, L. Du, Y. Dai, R. Wang, X. Liu, Y. Wu and S. Ji, SUB-PLAY: Adversarial Policies against Partially Observed Multi-Agent Reinforcement Learning Systems, {\it Proc. ACM SIGSAC Conf. Comput. Commun. Secur.}, 2024, pp 645-659.
    
    \bibitem{liu2022sample}
    Q. Liu, C. Szepesvári and C. Jin, Sample-Efficient Reinforcement Learning of Partially Observable Markov Games, {\it Adv. Neural Inf. Process. Syst.}, vol. 35, 2022, pp 18296-18308.
    
    \bibitem{kurniawati2009sarsop}
    H. Kurniawati, D. Hsu and W. Lee, {\it SARSOP: Efficient Point-Based POMDP Planning by Approximating Optimally Reachable Belief Spaces}, 2009.
    
    \bibitem{roy2005finding}
    N. Roy, G. Gordon and S. Thrun, Finding Approximate POMDP Solutions through Belief Compression, {\it J. Artif. Intell. Res.}, vol. 23, 2005, pp 1-40.
    
    \bibitem{kim2019pomhdp}
    S. K. Kim, O. Salzman and M. Likhachev, POMHDP: Search-Based Belief Space Planning using Multiple Heuristics, {\it Proc. Int. Conf. Autom. Plan. Sched.}, vol. 29, 2019, pp 734-744.

    \bibitem{lipp2016antagonistic}
    T. Lipp and S. Boyd, Antagonistic Control, {\it Syst. Control Lett.}, vol. 98, 2016, pp 44-48.
    \bibitem{taskesen2024distributionally}
    B. Taskesen, D. Iancu, C. Koçyiğit and D. Kuhn, Distributionally Robust Linear Quadratic Control, {\it Adv. Neural Inf. Process. Syst. (NeurIPS)}, vol. 36, 2024.
    
    \bibitem{hakobyan2024wasserstein}
    A. Hakobyan and I. Yang, Wasserstein Distributionally Robust Control of Partially Observable Linear Stochastic Systems, {\it IEEE Trans. Autom. Control}, 2024.
    
    \bibitem{moon2016linear}
    J. Moon and T. Başar, Linear Quadratic Risk-Sensitive and Robust Mean Field Games, {\it IEEE Trans. Autom. Control}, vol. 62, 2016, pp 1062-1077.
    
    \bibitem{bauso2016robust}
    D. Bauso, H. Tembine and T. Başar, Robust Mean Field Games. {\it Dyn. Games Appl.}, vol. 6, 2016, pp 277-303.
    
    \bibitem{lenhart2007optimal}
    S. Lenhart and J. Workman, {\it Optimal Control Applied to Biological Models}, Chapman, 2007.
    
    \bibitem{mcAsey2012convergence}
    M. McAsey, L. Moua and W. Han, Convergence of the Forward-Backward Sweep Method in Optimal Control, {\it Comput. Optim. Appl.}, vol. 53, 2012, pp 207-226.
    
    \bibitem{ng2000algorithms}
    A. Y. Ng and S. Russell, Algorithms for Inverse Reinforcement Learning, {\it ICML}, vol. 1, 2000, pp 2.
    
    
    \bibitem{Cook2021}
    J. A. Sharp, K. Burrage and M. J. Simpson, Implementation and Acceleration of Optimal Control for Systems Biology, {\it J. R. Soc. Interface}, vol. 18, 2021.

    \bibitem{Rose2015}
    G. R. Rose, Numerical Methods for Solving Optimal Control Problems, {\it M.Sc. Thesis}, University of Tennessee, Knoxville, 2015.


    \bibitem{han2016deep}
    J. Han and W. E, Deep Learning Approximation for Stochastic Control Problems, {\it arXiv Preprint arXiv:1611.07422}, 2016.


\bibitem{zhou2025integrating}
H. Zhou, D. Ralston, X. Yang, and R. Hu, Integrating Sequential Hypothesis Testing into Adversarial Games: A Sun Zi-Inspired Framework, {\it arXiv preprint arXiv:2502.13462}, 2025. Accepted for publication in the \textit{Proceedings of the 64th IEEE Conference on Decision and Control}.



\bibitem{ward2023active}
W. Ward, Y. Yu, J. Levy, N. Mehr, D. Fridovich-Keil, and U. Topcu, Active Inverse Learning in Stackelberg Trajectory Games, {\it arXiv preprint arXiv:2308.08017}, 2023.

\bibitem{kim2025deceptive}
Y. Kim, A. Benvenuti, B. Chen, M. Karabag, A. Kulkarni, N. D. Bastian, U. Topcu, and M. Hale, Deceptive Sequential Decision-Making via Regularized Policy Optimization, {\it arXiv preprint arXiv:2501.18803}, 2025.

\bibitem{mangasarian1966sufficient}
O. L. Mangasarian, Sufficient Conditions for the Optimal Control of Nonlinear Systems, {\it SIAM Journal on control}, vol. 4, 1966, pp 139-152.

\bibitem{wanner1996solving}
G. Wanner, and E. Hairer,
{\it Solving Ordinary Differential Equations II}, Springer Berlin Heidelberg, vol. 375, 1996.

\bibitem{das2024almost}
S. Das, P. Dey, and D. Chatterjee,
Almost Sure Detection of the Presence of Malicious Components in Cyber--Physical Systems,
{\it Automatica}, vol. 167, 2024.


 
\end{thebibliography}
\end{document}